\documentclass[onefignum,onetabnum]{siamart190516}
\usepackage{amsmath,amssymb,amsxtra}
\usepackage{mathrsfs}
\usepackage{graphicx}
\usepackage{algorithm, algpseudocode}

\usepackage{color,verbatim}

\allowdisplaybreaks
\numberwithin{equation}{section}

\allowdisplaybreaks

\title{A Stable Iterative Direct Sampling Method for Elliptic Inverse Problems with Partial Cauchy Data\thanks{The work of B. Jin is supported by Hong Kong RGC General Research Fund (14306423 and 14306824) and ANR / Hong Kong RGC Joint Research Scheme (A-CUHK402/24) and a start-up fund from The Chinese University of Hong Kong. The work of J. Zou was substantially supported by the Hong Kong RGC General Research Fund (projects 14310324, 14306623 and 14306921) and NSFC / Hong Kong RGC Joint Research Scheme 2022/23 (project N\_CUHK465/22).}}

\author{Bangti Jin\thanks{Department of Mathematics, The Chinese University of Hong Kong, Shatin, N.T., Hong Kong (email: \texttt{b.jin@cuhk.edu.hk, fengruwang@cuhk.edu.hk, zou@math.cuhk.edu.hk})}\and
Fengru Wang\footnotemark[2] \and
Jun Zou\footnotemark[2]}

\begin{document}

\maketitle

\begin{abstract}
We develop a novel iterative direct sampling method (IDSM) for solving linear or nonlinear elliptic inverse problems with partial Cauchy data.
It integrates three innovations: a data completion scheme to reconstruct missing boundary information, a heterogeneously regularized Dirichlet-to-Neumann map to enhance the near-orthogonality of probing functions, and a stabilization-correction strategy to ensure the numerical stability.
The resulting method is remarkably robust with respect to measurement noise, is flexible with the measurement configuration, enjoys provable stability guarantee, and achieves enhanced resolution for recovering inhomogeneities.
Numerical experiments in electrical impedance tomography, diffuse optical tomography, and cardiac electrophysiology show its effectiveness in accurately reconstructing the locations and geometries of inhomogeneities.
\end{abstract}

\begin{keywords}
iterative direct sampling method, elliptic inverse problem, partial Cauchy data, stability analysis
\end{keywords}


\section{Introduction} \label{sec:INT}

Inverse problems involving the estimation of physical parameters in elliptic partial differential equations (PDEs) from indirect and noisy measurements are foundational to many disciplines.
These problems aim to extract critical information of inhomogeneities inside the system, with broad applications in biomedical imaging \cite{Ammari2008}, nano-optics \cite{Novotny2006}, geophysics \cite{Zhdanov2015}, and environmental science \cite{ElBadia2002,Yakowitz1980} etc. {Most existing methods are designed for the case when full boundary data is available, and may work  poorly with partial data.
However, in many applications, it is often possible to collect only partial boundary data. The focus of the present work is to construct robust numerical methods for elliptic inverse problems with partial Cauchy data.}

Over the past three decades, direct methods have been developed for obtaining qualitative estimates of inhomogeneities.
Prominent methods include linear sampling method \cite{Kirsch1996}, MUltiple SIgnal Classification (MUSIC) \cite{AmmariCalmon:2008,Devaney2004,AmmariLesselier:2005}, point source method \cite{Potthast1998}, factorization method \cite{Kirsch1998}, and direct sampling method (DSM) \cite{ItoJinZou:2012}.
These techniques construct suitable indicator functions using functionals of measured data to indicate the presence of inhomogeneities, among which the DSM is notable due to its high computational efficiency and low data requirements.
It constructs index functions by computing the inner products of measured data and tailored probing functions. The DSM can achieve qualitative reconstructions of material inclusions using a few datasets.
It has been applied to diverse elliptic inverse problems, e.g., electrical impedance tomography \cite{Zou2014}, diffusion-based optical tomography \cite{Zou2015} and object detection \cite{Sun2023}.
By building on the conventional DSM, {in \cite{Ito2025Iterative}} we have proposed an Iterative DSM (IDSM) that achieves enhanced accuracy, improved stability, and broader applicability with only a modest increase in computational cost.

However, the extension of DSM and IDSM frameworks to scenarios involving partial Cauchy data remains challenging, which involves three fundamental challenges.
First, partial data disrupts the mutual orthogonality between probing functions and Green's function, which is essential for constructing accurate index functions in the full-data setting.
Second, the transition between accessible and inaccessible boundary segments induces additional discontinuity in the state, and complicates both mathematical analysis and numerical stability.
Third, existing formulations lack rigorous theoretical guarantees of the stability of the iterative process.
{The stability issue is particularly acute for partial boundary data since it typically requires more iterations to achieve satisfactory reconstructions.}
This represents a critical gap in the theoretical foundation of these methods.

To address these limitations, we propose a novel IDSM that integrates three key innovations.
First, we employ a data completion scheme to reconstruct the missing boundary data on the inaccessible portion. 
Second, we employ a heterogeneously regularized Dirichlet-to-Neumann map, which induces a dual product that enhances the near-orthogonality of the probing functions and ensures robustness to noisy data.
Third, by building on the dynamic low-rank structure of the existing IDSM framework \cite{Ito2025Iterative}, we incorporate a stabilization-correction updating scheme.
This scheme adaptively damps inaccurate low-rank terms from early iterations, and greatly enhances the accuracy of the resolver operator.
Moreover, we establish the stability of the IDSM, which fills a critical gap in the mathematical foundation of the DSM for elliptic inverse problems.

The rest of the paper is organized as follows.
In Section \ref{sec:PRO}, we formulate the mathematical model of elliptic inverse problems.
In Section \ref{sec:MAT}, we describe the IDSM framework with full boundary data.
In Section \ref{sec:pbm}, we develop the IDSM for partial Cauchy data by introducing three key innovations.
In Section \ref{sec:IMP}, we describe the practical implementation of the IDSM.
We present numerical experiments in Section \ref{sec:NUM}, and give concluding remarks in Section \ref{sec:CON}.
Throughout, for any Hilbert space $H$ and Banach space $X$, the notation $(\cdot, \cdot)_H$ denotes the inner product on $H$, while $\langle \cdot, \cdot \rangle_{X', X}$ denotes the duality pairing between $X$ and its dual space $X'$.
For Sobolev spaces $X$ and $Y$, we denote by $B(X,Y)$ the space of bounded operators from $X$ to $Y$, with $L(X,Y) \subset B(X,Y)$ denoting the subspace of bounded linear operators, and $X \hookrightarrow Y$ denoting the Sobolev embedding.

\section{Problem formulation}\label{sec:PRO}
Let $\Omega\subset \mathbb{R}^{d}$ ($d=2,3$) be an open, bounded, simply connected domain with a smooth boundary $\Gamma=\partial\Omega$.
The boundary $\Gamma$ is divided into two parts: the accessible part $\Gamma_{D}\subset \Gamma$ and the inaccessible part $\Gamma_{N}=\Gamma\setminus\Gamma_{D}$.
The unknown parameter $u$ represents inhomogeneous inclusions whose material properties differ from the background medium. It belongs to the admissible set $S = \{u{\in L^{\infty}(\Omega)} : a \leq u \leq b\}$ (with $a < b$), which enforces physically relevant constraints.
We assume that $u$ admits the representation
\begin{equation}\label{eqn2}
u(x) = \sum_{i=1}^n c_i \chi_{\omega_i}(x),
\end{equation}
where $n$ denotes the number of inclusions, $\{\omega_j\}_{j=1}^n$ are pairwise disjoint subdomains compactly contained in $\Omega$, $\chi_{\omega_i}$ is the characteristic function of $\omega_i$, and each $c_i \in \mathbb{R}$ is a constant within $\omega_i$.
In the presence of multiple types of inhomogeneities, the function $u$ is vector-valued with each component taking the form \eqref{eqn2}.
For example, in diffuse optical tomography (DOT),
\begin{equation}\label{eqn3}
\left\{\begin{aligned}
\nabla \cdot ((u_{c}+c_{0}) \nabla y) + (u_{p}+p_{0}) y &= 0, \quad \text{in } \Omega, \\
c_{0} \partial_{n}y &= f, \quad \text{on } \Gamma,
\end{aligned}\right.
\end{equation}
where $u_c$ and $u_{p}$ are the deviations of the conductivity and potential coefficient from the background medium, respectively{,   and both $u_c$ and $u_p$ take the form \eqref{eqn2}.}

The direct problem of the (possibly nonlinear) elliptic problem in an operator form reads:
\begin{equation}\label{eqn1}
\mathcal{A}[y]y + \mathcal{B}[u](y) = f,
\end{equation}
where $f \in L^2(\Gamma) \hookrightarrow (H^1(\Omega))'$ denotes a boundary source (via the trace theorem), or $f \in L^2(\Omega) \hookrightarrow (H^1(\Omega))'$ denotes a domain source.
The elliptic operator $\mathcal{A}: H^1(\Omega) \to B(H^1(\Omega), (H^1(\Omega))')$ models the background physical process and may be semilinear in the state $y$.
The operator $\mathcal{B}$ linearly maps $S \subset L^\infty(\Omega)$ to $B(H^1(\Omega), (H^1(\Omega))')$, and describes the influence of the inhomogeneity $u$ on the state $y$, which may also be nonlinear in $y$.
When $\mathcal{A}[y]y$ and $\mathcal{B}[u](y)$ are linear in $y$, we write them as $\mathcal{A}y$ and $\mathcal{B}[u]y$.
For example, for the model \eqref{eqn3}, the operators $\mathcal{A}$ and $\mathcal{B}$ are defined by
\begin{equation*}
\begin{aligned}
\langle\mathcal{A}y,v\rangle_{(H^{1}(\Omega))', H^{1}(\Omega)} &= \int_{\Omega} c_0 \nabla y \cdot \nabla v + p_0 y v \, {\rm d}x, \\
\langle\mathcal{B}[u]y,v\rangle_{(H^{1}(\Omega))', H^{1}(\Omega)} &= \int_{\Omega} (u_{c},u_{p}) (\nabla y \cdot \nabla v , y v)^{\top} \, {\rm d}x.
\end{aligned}
\end{equation*}
For a detailed description of the abstract form \eqref{eqn1}, see \cite[Section 2]{Ito2025Iterative}.

The objective of the inverse problem is to identify the number, locations, and shapes of inclusions within the domain $\Omega$, using only a few pairs of partial Cauchy data.
The partial data setting exacerbates the inherent ill-posedness of the inverse problem, especially for $\omega_{j}$ far away from $\Gamma_{D}$.

\section{The IDSM for full boundary data}\label{sec:MAT}
{In this section we first describe the IDSM that we developed for the full boundary data in \cite{Ito2025Iterative}, and then discuss some challenges for the partial data.} To this end, we employ several operators.
The transpose operator $\mathcal{B}_\tau$ associated with $\mathcal{B}$ is defined by {(with the inhomogeneity $u$)}
\begin{equation}\label{eqn9}
\mathcal{B}_\tau[\cdot] : H^1(\Omega) \to L(L^\infty(\Omega), (H^1(\Omega))'), \quad
\mathcal{B}_\tau[y]u = \mathcal{B}[u](y).
\end{equation}
The adjoint operator $\mathcal{B}_\tau[y]^{*}\in L(H^{1}(\Omega),(L^{\infty}(\Omega))')$ is defined via duality pairing:
\begin{equation*}
\langle \mathcal{B}_{\tau}[y]^{*} p, u \rangle_{(L^\infty(\Omega))', L^\infty(\Omega)} =\langle \mathcal{B}_{\tau}[y]u, p \rangle_{(H^{1}(\Omega))', H^{1}(\Omega)} = \langle \mathcal{B}[u](y), p \rangle_{(H^{1}(\Omega))', H^{1}(\Omega)}.
\end{equation*}
These identities characterize the relationship between $\mathcal{B}$, its transpose $\mathcal{B}_\tau$, and the adjoint $\mathcal{B}_\tau^{*}$.
We denote by $u_{*}$ the ground truth inclusion, and by $y(u)$ the solution to the elliptic system \eqref{eqn1} with $u$.
The trace operator $\mathcal{T} \in L(H^1(\Omega), L^2(\Gamma))$ maps $H^1(\Omega)$ functions to their boundary traces in $L^2(\Gamma)$, and the restricted version $\mathcal{T}_D \in L(H^1(\Omega), L^2(\Gamma))$ maps to the boundary $\Gamma_D$:
\begin{equation*}
(\mathcal{T}_D y)(x) =
\begin{cases}
y(x), & x \in \Gamma_D, \\
0, & \text{otherwise}.
\end{cases}
\end{equation*}
$\mathcal{T}_{D}$ and its complement $\mathcal{T}_N = \mathcal{T} - \mathcal{T}_D$ (supported on $\Gamma_N = \Gamma \setminus \Gamma_D$) provide a partition of the function in $L^{2}(\Gamma)$.
The choice of $L^2(\Gamma)$ is to match the measured data $y_d \in L^2(\Gamma)$, which is a noisy version of $\mathcal{T}_{D}y(u_{*})$, i.e., $y_d = \mathcal{T}_D y(u_*) + \varepsilon$ with $\varepsilon\in L^{2}(\Gamma)$ { being the noise}.

We now connect the unknown inclusion parameter $u_*$ and the state $y(u_*)$.
From \eqref{eqn1}, the inclusion $u_{*}$ and the corresponding state $y(u_{*})$ satisfy
\begin{equation}\label{eqn5}
y(u_{*}) = \mathcal{A}[y(u_{*})]^{-1}\left(f - \mathcal{B}[u_{*}](y(u_{*}))\right).
\end{equation}
Let $y_{\emptyset}(u)$ be the background solution (i.e., without the inclusion):
\begin{equation}\label{eqn6}
y_{\emptyset}(u) = \mathcal{A}[y(u)]^{-1}f.
\end{equation}
The scattering field $y^s(u)$ is defined by
\begin{equation*}
y^{s}(u) = \mathcal{T}y_{\emptyset}(u) - \mathcal{T}y(u_{*}).
\end{equation*}
By subtracting equation \eqref{eqn5} from the background equation \eqref{eqn6}, we obtain the expression for the scattering field:
\begin{align}
&y^{s}(u_{*}) = \mathcal{T}y_{\emptyset}(u_{*}) - \mathcal{T}y(u_{*}) \nonumber\\
=& \mathcal{T}\mathcal{A}[y(u_{*})]^{-1}f - \mathcal{T}\mathcal{A}[y(u_{*})]^{-1}\left(f - \mathcal{B}[u_{*}](y(u_{*}))\right)\nonumber \\
=& \mathcal{T}\mathcal{A}[y(u_{*})]^{-1}\mathcal{B}[u_{*}](y(u_{*})) = \mathcal{T}\mathcal{A}[y(u_{*})]^{-1}\mathcal{B}_{\tau}[y(u_{*})]u_{*}.\label{eqn8}
\end{align}
This motivates the forward operator $\mathcal{H}[u] = \mathcal{T}\mathcal{A}[y(u)]^{-1}\mathcal{B}_{\tau}[y(u)]$, which maps $S\subset L^{\infty}(\Omega)$ linearly to $L^{2}(\Gamma)$ and relates the inclusion parameter $u$ to the scattering field $y^{s}(u)$.

Then recovering the inclusion $u$ amounts to inverting the equation $y^{s}(u_{*}) = \mathcal{H}[u_{*}]u_{*}$.
{The IDSM constructs an index function $\eta$ that approximates $u_{*}$ through a probing operator $\mathcal{K}$:
\begin{equation*}
\eta = \mathcal{K}\mathcal{H}[u_{*}]u_{*}.
\end{equation*}
That is, $\eta$ is defined pointwise via a probing kernel $K(x_b, x):\Gamma\times\Omega\mapsto\mathbb{R}$ that generates probing functions $k_x(x_b) := K(x_b, x)$:}
\begin{equation*}
\eta(x) = (\mathcal{K}y^{s}(u_{*}))(x) = (k_x, y^{s}(u_{*}))_{L^{2}(\Gamma)}.
\end{equation*}
If $\mathcal{K}$ approximately inverts $\mathcal{H}[u_{*}]$, i.e., $\mathcal{K}\mathcal{H}[u_{*}] \approx \mathcal{I}_{L^{\infty}(\Omega)}$ (the identity operator on $L^\infty(\Omega)$), then $\eta$ provides a qualitative reconstruction of the inclusion $u_{*}$:
\begin{equation*}
\eta = \mathcal{K}\mathcal{H}[u_{*}]u_{*} \approx u_{*}.
\end{equation*}
When constructing the probing operator $\mathcal{K}$, a lifting operator is required to map the scattering field $y^{s}$ from the boundary $\Gamma$ to the domain $\Omega$.
One common choice is the adjoint operator $\mathcal{H}[u_{*}]^{*} \in \mathcal{L}(L^{2}(\Gamma), (L^{\infty}(\Omega))')$.
The $L^{2}(\Gamma)$ adjoint of $\mathcal{H}[u_{*}]$ is defined via duality pairing:
$$\langle \mathcal{H}[u_{*}]^{*} y, v \rangle_{(L^{\infty}(\Omega))', L^{\infty}(\Omega)} = (y, \mathcal{H}[u_{*}]v)_{L^{2}(\Gamma)}, \quad\forall v \in L^{\infty}(\Omega).$$
This leads to the dual function:
\begin{equation}\label{eqn7}
\zeta = \mathcal{H}[u_{*}]^{*} y^{s} = \mathcal{G}[u_{*}]u_{*},
\end{equation}
where $\mathcal{G}[u] = \mathcal{H}[u]^{*}\mathcal{H}[u] \in \mathcal{L}(L^{\infty}(\Omega), (L^{\infty}(\Omega))')$ denotes the Gramian of $\mathcal{H}[u]$.
{
We then employ a resolver $\mathcal{R} \in \mathcal{L}((L^{\infty}(\Omega))', L^{\infty}(\Omega))$ that approximates the inverse of $\mathcal{G}[u_{*}]$.
Together with the adjoint operator $\mathcal{H}[u_{*}]^{*}$, this induces a probing operator $\mathcal{K} = \mathcal{R}\mathcal{H}[u_{*}]^{*}$, which approximates the inverse of $\mathcal{H}[u_{*}]$.
}

In practice, there are two limitations in computing the index function $\eta$ due to the nonlinearity of the inverse problem.
First, the knowledge of $y(u_*)$ is limited to the boundary $\Gamma$, and thus applying $\mathcal{H}[u_*]^*$ in \eqref{eqn7} is intractable.
Second, when the operator $\mathcal{A}$ is nonlinear, $\mathcal{A}[y(u_*)]$ is unknown.
To overcome the nonlinearity, the IDSM employs a fixed-point scheme:
\begin{equation}\label{eqn11}
\eta^{k+1} =\mathcal{R}\zeta^{k}= \mathcal{R}\mathcal{H}[u^k]^* y^s(u^k).
\end{equation}
Given the initial guess $u^0 = 0$, the index function $\eta^{k+1}$ at the $(k+1)$th iteration is computed via
\begin{equation*}
\eta^{k+1} = \mathcal{R} \left( \mathcal{T}\mathcal{A}[y(u^k)]^{-1}\mathcal{B}_\tau[y(u^k)] \right)^* \left( \mathcal{T}y_\emptyset(u^k) - \mathcal{T}y(u_{*}) \right).
\end{equation*}
The updated inclusion estimate $u^{k+1}$ is obtained by projecting $\eta^{k+1}$ onto the admissible set $[a, b]$ via the pointwise projection operator $\mathcal{P}_{[a,b]}$, and the corresponding potential $y(u^{k+1})$ is computed by solving the forward problem \eqref{eqn1}.
The iteration then proceeds with the new pair $(u^{k+1}, y(u^{k+1}))$.
This process progressively refines the estimate of $u^k$ and drives the potential $y(u^k)$ toward $y(u_*)$, thereby addressing the nonlinearity of the problem.

The extension of the IDSM framework to the setting with partial Cauchy data faces three challenges. (i) The lack of full boundary measurements leads to uncertainty in the state trace $\mathcal{T}y(u_*)$; (ii) The near-orthogonality of Green’s functions, which is essential for constructing accurate index functions, is compromised. This degrades the resolution, particularly in the regions far away from $\Gamma_D$; (iii) The ill-posedness of the problem under partial data exacerbates the unboundedness of  $\mathcal{R}$, which aggravates the numerical instability over successive iterations.

\section{The IDSM for partial boundary measurement}\label{sec:pbm}
In this section, we develop the new IDSM for partial Cauchy data based on three key innovations: data completion, heterogeneously regularized DtN map and stabilization-correction scheme.
\subsection{Data completion}\label{subsec:dc}

To address the challenge of partial data, we employ a data completion procedure that leverages the available Cauchy data on $\Gamma_D$ to reconstruct the missing data on $\Gamma_N$.
We describe the approach for diffuse optical tomography (DOT) (cf. \eqref{eqn3}).

Consider the inverse operator $\mathcal{A}^{-1}$, which acts through Green's function $G(x,x')$:
\begin{equation*}
\left\{
\begin{aligned}
- \nabla_x \cdot \left( c_0 \nabla_x G(x,x') \right) + p_0 G(x,x') &= \delta_{x-x'}, & \text{in } \Omega, \\
\partial_{n}G(x,x') &= 0, & \text{on } \Gamma.
\end{aligned}
\right.
\end{equation*}
$G(x,x')$ represents the action of $\mathcal{A}^{-1}$ via the integral operator $z(x) = \int_{\Omega} G(x,x') w(x') \, {\rm d}x'$.
The operator $\mathcal{B}_\tau$ defined in \eqref{eqn9} is given by
\begin{equation*}
\langle \mathcal{B}_\tau[y]u, v \rangle_{(H^1(\Omega))', H^1(\Omega)} = \langle\mathcal{B}[u]y,v\rangle_{(H^{1}(\Omega))', H^{1}(\Omega)} = \int_{\Omega} [u_{c},u_{p}] [\nabla y \cdot \nabla v , y v]^{\top} {\rm d}x.
\end{equation*}
Consequently,
\begin{equation*}
\begin{aligned}
y^{s}(x) =& \left( \mathcal{T} \mathcal{A}^{-1} \mathcal{B}_\tau[y(u_*)]u_* \right)(x) = \int_{\Omega} G(x,x') \left[ u_{c*}(x') \nabla_{x'} y(u_*)(x') \cdot \nabla_{x'} + u_{a*}(x') y(u_*)(x') \right] {\rm d}x'\\
\approx &\sum_{i} w_{i} \left[ u_{c*}(x_{i}) y(u_*)(x_{i}) G(x_{i},x) + u_{a*}(x_{i}) \nabla y(u_*)(x_{i}) \cdot \nabla_{x_{i}} G(x_{i},x) \right],
\end{aligned}
\end{equation*}
where $w_{i}$ denotes quadrature weights { and $x_{i}$ are sampling points inside the support of the inclusion. This formal derivation aligns with the rigorous treatment in \cite[(3.12)]{Chow2021}.}
Thus, $y^{s}$ on the boundary $\Gamma_{N}$ can be expressed as a linear combination of the Dirichlet traces of $G(\cdot,x)$ and $\nabla G(\cdot,x)$, whose coefficients can be determined from the available measurements on $\Gamma_D$.

To compensate the missing data on $\Gamma_N$, we exploit the iterative estimate from the fixed-point scheme \eqref{eqn11}.
Given an approximate potential $y(u^{k})$, the boundary value on $\Gamma$ is given by
\begin{equation*}
y_{\emptyset}(u^{k})(x) - y(u^{k})(x) \approx \sum_{i} w_{i} \left[ u^{k}_{c}(x_{i}) y(u^{k})(x_{i}) G(x_{i},x) + u^{k}_{p}(x_{i}) \nabla y(u^{k})(x_{i}) \cdot \nabla_{x_{i}} G(x_{i},x) \right],
\end{equation*}
where the linear coefficients $u^{k}_{c}(x_{i}) y(u^{k})(x_{i})$ and $u^{k}_{p}(x_{i}) \nabla y(u^{k})(x_{i})$ are derived from the current estimate $u^{k}$, and approximate the exact ones $u_{c*}(x_{i}) y(u_{*})(x_{i})$ and $u_{a*}(x_{i}) \nabla y(u_{*})(x_{i})$.
This motivates the the following data completion formula:
\begin{equation}\label{eqn31}
\tilde{y}_{d}(u^{k}) = y_{d} + \mathcal{T}_N y(u^{k})=\mathcal{T}_{D}y(u_{*})+ \mathcal{T}_N y(u^{k})+\varepsilon.
\end{equation}
At each iteration, the scattering field is given by $\tilde{y}^{s}_{d}(u^{k})=\mathcal{T}y_{\emptyset}(u^{k})-\tilde{y}_{d}(u^{k})$.
Then the index function is computed by
\begin{equation*}
\eta^{k+1}=\mathcal{R}\mathcal{H}[u^{k}]^{*}\tilde{y}^{s}_{d}(u^{k})=\mathcal{R}\mathcal{H}[u^{k}]^{*}\left(\mathcal{T}y_{\emptyset}(u^{k})-y_{d} - \mathcal{T}_N y(u^{k})\right).
\end{equation*}
This data completion procedure leverages the current state estimate $y(u^k)$ to approximate the missing boundary condition on $\Gamma_N$.
However, since the estimate $u^k$ (and also $y(u^k)$) may deviate from the true solution, especially in early iterations, the completed data $\mathcal{T}_N y(u^k)$ on $\Gamma_N$ is only approximate.
The error necessitates incorporating specialized regularization strategies to mitigate its adverse impact on the reconstruction stability.

\subsection{HR-DtN map}\label{subsec:hrdm}
To address the challenge of the ill-posedness in $\mathcal{G}[u]$, we now propose a novel heterogeneously regularized Dirichlet-to-Neumann (HR-DtN) map $\Lambda_{\alpha,D}(\mathcal{A})$.
It is designed to simultaneously enhance the accuracy of the reconstruction, suppress noise on the accessible boundary $\Gamma_D$, and mitigate the errors from the data completion procedure on $\Gamma_N$.

The inherent ill-posedness of inverting $\mathcal{G}[u]$ stems from its failure to induce a dual product that can distinguish inclusions with disjoint supports.
Thus, the dual function $\zeta = \mathcal{G}[u]u$ exhibits low sensitivity to the spatial distribution of the inclusion $u$.
{
This is quantified by the cross term: for $u_1, u_2 \in L^{\infty}(\Omega)$ with disjoint supports, $\langle \mathcal{G}[u]u_1, u_2 \rangle$ decays only slowly, whereas it should decay rapidly with the separation distance in order to resolve distinct inclusions.}
The inability of $\mathcal{G}[u]$ to make the cross term small when $u_1$ and $u_2$ have disjoint supports limits its resolving power.
To address the issue, prior studies \cite{Zou2015,Zou2014} replace the $L^2(\Gamma)$-adjoint $\mathcal{H}[u]^*$ with an $H^{\mu}(\Gamma)$-adjoint $\mathcal{H}_{\mu}[u]^*$, by exploiting the near-orthogonality relation 
\begin{equation*}
\frac{(G(x,\cdot), G(x',\cdot))_{H^{\mu}(\Gamma)}}{\|G(x,\cdot)\|_{H^{\mu}(\Gamma)}\|G(x',\cdot)\|_{H^{\mu}(\Gamma)}} \approx
\begin{cases}
1, & \text{if $x$ is close to $x'$,} \\
\ll 1, & \text{otherwise}.
\end{cases}
\end{equation*}
{With the $H^{\mu}(\Gamma)$-adjoint, we can formally write for $u_{1}$, $u_{2}$ with disjoint supports that}
\begin{equation*}
\langle \mathcal{G}_{\mu}[u]u_1, u_2 \rangle_{(L^\infty(\Omega))', L^\infty(\Omega)} \triangleq\langle \mathcal{H}_{\mu}[u]^* \mathcal{H}[u]u_{1},u_{2}\rangle_{(L^{\infty}(\Omega))',L^{\infty}(\Omega)}\triangleq (\mathcal{H}[u]u_1, \mathcal{H}[u]u_2)_{H^{\mu}(\Gamma)} \approx 0.
\end{equation*}
That is, it leads to negligible cross term with respect to $\mathcal{G}_{\mu}[u]$, which enhances the sensitivity to inclusion distributions.
However, the $H^{\mu}(\Gamma)$ inner product involves differentiating the measured data, which induces instability in the presence of noise.
{We proposed a regularized DtN map $\Lambda_{\alpha}(\mathcal{A})$ in \cite{Ito2025Iterative}, }
which interpolates between the $H^{1/2}(\Gamma)$-inner product (as $\alpha \to 0^+$) and the $L^{2}(\Gamma)$-inner product (as $\alpha \to \infty$).
This achieves an optimal balance, enhancing the sensitivity of the index function to inclusion location while maintaining robustness to data noise.

However, $\Lambda_{\alpha}(\mathcal{A})$ with a uniform regularization parameter $\alpha$ is not directly applicable in partial data setting.
The completed data $\tilde{y}_{d}(u^{k}) = y_{d} + \mathcal{T}_{N}y(u^{k})$ exhibits heterogeneous fidelity: $\tilde{y}_{d}(u^{k})|_{\Gamma_{D}} = y_d$ is the actual measurements, but $\tilde{y}_{d}(u^{k})|_{\Gamma_{N}} = \mathcal{T}_{N}y(u^{k})$ is an approximation whose accuracy depends on the current iterate $u^{k}$.
To accommodate the disparity, we employ a spatially dependent regularization parameter $\alpha_D = \alpha_{d}\chi_{\Gamma_{D}}+\alpha_{n}\chi_{\Gamma_{N}}$, where $0 < \alpha_{d} \ll \alpha_{n}$.
{We will explain this choice later on.}
For a given $v \in L^2(\Gamma)$ and the elliptic operator $\mathcal{A} \in \mathcal{L}(H^1(\Omega), (H^1(\Omega))')$, the HR-DtN map $\Lambda_{\alpha,D}(\mathcal{A})v = p$ is defined by the pair $(w, p) \in H^1(\Omega) \times L^2(\Gamma)$ satisfying
\begin{equation}\label{eqn12}
\left\{
\begin{aligned}
\langle \mathcal{A}w, z \rangle_{(H^1(\Omega))', H^1(\Omega)} - (p, \mathcal{T}z)_{L^2(\Gamma)} &= 0, \quad &\forall z \in H^1(\Omega), \\
(\mathcal{T}w, q)_{L^2(\Gamma)} + (\alpha_D p, q)_{L^2(\Gamma)} &= (v, q)_{L^2(\Gamma)}, \quad &\forall q \in L^2(\Gamma).
\end{aligned}
\right.
\end{equation}
The following lemma gives the well-posedness of $\Lambda_{\alpha,D}(\mathcal{A})$ and a  stability estimate.

\begin{lemma}\label{lemma1}
Let the elliptic operator $\mathcal{A} \in \mathcal{L}(H^{1}(\Omega), H^{1}(\Omega)')$ satisfy the following conditions:
\begin{equation*}
\begin{aligned}
m \|w\|_{H^{1}(\Omega)}^{2} \leq \langle \mathcal{A}w, w \rangle_{(H^{1}(\Omega))', H^{1}(\Omega)},&\quad\forall w\in H^{1}(\Omega), \\
\langle \mathcal{A}w, z \rangle_{(H^{1}(\Omega))', H^{1}(\Omega)} \leq M \|w\|_{H^{1}(\Omega)} \|z\|_{H^{1}(\Omega)},&\quad\forall w,z\in H^{1}(\Omega),
\end{aligned}
\end{equation*}
where $m, M > 0$ are constants.
Then, there exists a unique solution $p \in L^{2}(\Gamma)$ of \eqref{eqn12} such that
\begin{equation*}
\alpha_{d} \|p\|_{L^{2}(\Gamma_{D})}^{2} + \alpha_{n} \|p\|_{L^{2}(\Gamma_{N})}^{2} \leq \alpha_{d}^{-1}\|v\|_{L^{2}(\Gamma_{D})}^{2} + \alpha_{n}^{-1} \|v\|_{L^{2}(\Gamma_{N})}^{2}.
\end{equation*}
\end{lemma}

\begin{proof}
Setting $\tilde{p} = \sqrt{\alpha_{D}} p$ transforms problem \eqref{eqn12} into
\begin{equation*}
\left\{
\begin{aligned}
\langle \mathcal{A}w, z\rangle_{(H^{1}(\Omega))', H^{1}(\Omega)} - ( \alpha_D^{-1/2}\tilde{p}, \mathcal{T}z )_{L^{2}(\Gamma)} &= 0, & \forall z \in H^{1}(\Omega), \\
(\mathcal{T}w, q)_{L^{2}(\Gamma)} + (\sqrt{\alpha_{D}} \tilde{p}, q)_{L^{2}(\Gamma)} &= (v, q)_{L^{2}(\Gamma)}, & \forall q \in L^{2}(\Gamma),
\end{aligned}
\right.
\end{equation*}
which can be rewritten as
\begin{equation}\label{eqn32}
\left\{
\begin{aligned}
\langle \mathcal{A}w, z\rangle_{(H^{1}(\Omega))', H^{1}(\Omega)} - (\alpha_D^{-1/2} \tilde{p}, \mathcal{T}z )_{L^{2}(\Gamma)} &= 0, & \forall z \in H^{1}(\Omega), \\
( \alpha_D^{-1/2}\mathcal{T}w, \tilde{q})_{L^{2}(\Gamma)} + (\tilde{p}, \tilde{q})_{L^{2}(\Gamma)} &= (\alpha_D^{-1/2} v, \tilde{q} )_{L^{2}(\Gamma)}, & \forall \tilde{q} = \sqrt{\alpha_{D}} q \in L^{2}(\Gamma).
\end{aligned}
\right.
\end{equation}
It defines a variational problem over the product space $H^{1}(\Omega) \times L^{2}(\Gamma)$.
We define a bilinear form $a: [H^{1}(\Omega) \times L^{2}(\Gamma)]^{2} \to \mathbb{R}$ and a linear functional $b: H^{1}(\Omega) \times L^{2}(\Gamma) \to \mathbb{R}$ respectively by
\begin{align*}
a\left( [w, \tilde{p}], [z, \tilde{q}] \right) =& \langle \mathcal{A}w, z \rangle_{(H^{1}(\Omega))', H^{1}(\Omega)} - (\alpha_D^{-1/2}\tilde{p}, \mathcal{T}z)_{L^{2}(\Gamma)} + (\alpha_D^{-1/2} \mathcal{T}w, \tilde{q} )_{L^{2}(\Gamma)} + (\tilde{p}, \tilde{q})_{L^{2}(\Gamma)}, \\
b\left( [z, \tilde{q}] \right) =& (\alpha_D^{-1/2} v, \tilde{q} )_{L^{2}(\Gamma)}.
\end{align*}
The continuity of $\mathcal{A}$ and $\mathcal{T}$ ensures that $a$ and $b$ are continuous.
The coercivity of $a$ follows by
\begin{align*}
a([w, \tilde{p}], [w, \tilde{p}]) =& \langle \mathcal{A}w, w \rangle_{(H^{1}(\Omega))', H^{1}(\Omega)} - (\alpha_D^{-1/2} \tilde{p}, \mathcal{T}w )_{L^{2}(\Gamma)} + (\alpha_D^{-1/2} \mathcal{T}w, \tilde{p} )_{L^{2}(\Gamma)} + \|\tilde{p}\|_{L^{2}(\Gamma)}^{2} \\
=& \langle \mathcal{A}w, w \rangle_{(H^{1}(\Omega))', H^{1}(\Omega)} + \|\tilde{p}\|_{L^{2}(\Gamma)}^{2}\geq m \|w\|_{H^{1}(\Omega)}^{2} + \|\tilde{p}\|_{L^{2}(\Gamma)}^{2}.
\end{align*}
By Lax-Milgram theorem, there exists a unique solution $[w, \tilde{p}] \in H^{1}(\Omega) \times L^{2}(\Gamma)$ to the variational problem \eqref{eqn32}, implying the uniqueness of $p \in L^{2}(\Gamma)$.
For the stability estimate, we observe
\begin{equation*}
\|\tilde{p}\|_{L^{2}(\Gamma)}^{2} \leq a([w, \tilde{p}], [w, \tilde{p}]) = b([w, \tilde{p}]).
\end{equation*}
The definition of $b$ and the Cauchy-Schwarz inequality imply $$\|b\|_{(H^{1}(\Omega) \times L^{2}(\Gamma))'} \leq \|\alpha_D^{-1/2} v\|_{L^{2}(\Gamma)}.$$
This, the coercivity bound and the substitution $\tilde{p} = \sqrt{\alpha_{D}} \, p$ yield the desired assertion.
\end{proof}

We adopt the adjoint operator $\mathcal{H}_{D}[u]^{*}$ as the lifting operator in the IDSM:
\begin{equation*}
\langle \mathcal{H}_{D}[u]^{*} y, v \rangle_{(L^\infty(\Omega))', L^\infty(\Omega)} = \left( \Lambda_{\alpha,D}(\mathcal{A}[y(u)]) y, \Lambda_{\alpha,D}(\mathcal{A}[y(u)]) \mathcal{H}[u] v \right)_{L^2(\Gamma)},\,\forall y,v \in L^{2}(\Gamma)\times L^\infty(\Omega),
\end{equation*}
where the choice of $\mathcal{A}[y(u)]$ is to reduce the computation cost, cf. Lemma \ref{lemma2}.
We denote by $\mathcal{G}_{D}[u]$ its Gramian, i.e., $\mathcal{G}_{D}[u] = \mathcal{H}_{D}[u]^{*}\mathcal{H}[u]$.
The adjoint operator $\mathcal{H}_D[u]^*$ offers distinct advantages over existing approaches in the partial data setting.
First, it preserves the near-orthogonality of the probing functions on the accessible boundary $\Gamma_D$ with small $\alpha_{d}$.
Second, it delivers superior noise robustness compared to the $H^{\mu}(\Gamma)$-adjoint, by  avoiding explicitly differentiating the measured data.
Most significantly, the choice $(\alpha_{d}, \alpha_{n})$ accommodates  the heterogeneity of data quality.
A small $\alpha_{d}$ on $\Gamma_D$ retains data fidelity (i.e., high-frequency information in the measurements), and a large $\alpha_{n}$ on $\Gamma_N$ suppresses errors and uncertainties of the completed data.
Finally, the design is flexible with parameter selection and measurement configuration{, since the parameters can be adapted to different boundary measurement configurations.}

\subsection{Stabilization-correction scheme}\label{subsec:sroc}

The construction of the resolver $\mathcal{R}$ is central to the IDSM, which impacts greatly the quality of the index function $\eta${; see section \ref{sec:MAT} for the case with full boundary data}.
An effective $\mathcal{R}$ must fulfill three  requirements: (i) computational efficiency to ensure practicality within an iterative framework; (ii) high accuracy in $\mathcal{R}\mathcal{G}_{D}[u] \approx \mathcal{I}_{L^{\infty}(\Omega)}$; (iii) noise stability.
{
We developed an efficient approach to construct $\mathcal{R}$ in \cite{Ito2025Iterative} when full boundary data is available. 
Now we develop a novel and stable scheme for constructing $\mathcal{R}$ that satisfies these criteria when only partial boundary data is available.}
Furthermore, we  establish the uniform boundedness of the resulting operator, which guarantees the robustness of the overall iterative process.

{For computational efficiency, the DSM \cite{Zou2014,Zou2015,Chow2021} employs a delta-type resolver $\mathcal{R}=D\delta$:
\begin{equation}\label{eqn13}
\eta(x) = (\mathcal{R}^{0}\zeta)(x)=\int_{x'}D(x)\delta(x'-x)\zeta(x')\,{\rm d}x = D(x)\zeta(x).
\end{equation}}
In the full-data case, $D(x)$ is commonly constructed as $D(x) = \|G(x,\cdot)\|_{H^{\mu}(\Gamma)}^{-2}$.
This choice is motivated by the formal duality pairing $\langle \delta_{x},\mathcal{G}_{\mu}[u]\delta_{x}\rangle_{(L^{\infty}(\Omega))',L^{\infty}(\Omega)} = \|G(x,\cdot)\|_{H^{\mu}(\Gamma)}^{2}$.
Asymptotically, $G(x, x')$ behaves like the fundamental solution $\Phi_x(x')$ associated with $\mathcal{A}$ (e.g., $\Phi_x(x') = -\frac{1}{2\pi}\ln|x - x'|$ for $\mathcal{A} = -\Delta$ in $\mathbb{R}^2$), satisfying $\lim_{x' \to x} G(x, x') / \Phi_x(x') = 1$.
It suggests the choice $D(x) = \|\Phi_x\|_{H^{\mu}(\Gamma)}^{-2}$.
Also the distance-to-boundary $d(x, \Gamma) = \inf_{x' \in \Gamma} |x - x'|$ can be used: $D(x) = d(x,\Gamma)^{\gamma}\approx G(x,x)^{-2}$, where the exponent $\gamma > 0$ is determined empirically.

In the partial data case, the non-uniform behavior of $\Lambda_{\alpha,D}(\mathcal{A})$ precludes the use of a diagonal $D(x)$.
$\Lambda_{\alpha,D}(\mathcal{A})$ adopts a heterogeneous regularization: a small $\alpha_{d}$ on $\Gamma_D$ endows the map with $(-\Delta_\Gamma)^{1/2}$-like behavior, while a large $\alpha_{n}$ on $\Gamma_N$ reduces to an $L^2(\Gamma)$ form.
Thus we define
\begin{equation}\label{eqn16}
D(x) = C_{D}\left\| \Phi_x \frac{\alpha_{D}}{1+\alpha_{D}} + |\nabla \Phi_x|\frac{1}{1+\alpha_{D}} \right\|_{L^2(\Gamma)}^{-\gamma}\chi_{\Omega^{\varepsilon}}(x),
\end{equation}
with $\Omega^{\varepsilon}=\{x\in \Omega:\,d(x,\Gamma)\geq \varepsilon\}$ {to ensure that $D(x)$ vanishes near the boundary $\Gamma$.}
This design leverages the asymptotic properties of $\Phi_x$ to emulate Green's function.
The scaling constant is $C_D = \|\hat{\eta}^{k+1}\|_{L^{1}(\Omega)} / \|D\hat{\zeta}^{k+1}\|_{L^{1}(\Omega)}$, where $\hat{\eta}^{k+1}$ and $\hat{\zeta}^{k+1}$ are the auxiliary index function and dual function, ensures the consistency with the current iteration.

While \eqref{eqn13} is computationally efficient, its accuracy is limited.
This motivates the use of a hybrid formulation for $\mathcal{R}^{k} = D\delta + \sum_{j=1}^{K} n_{j} \otimes m_{j}$ in the IDSM \cite{Ito2025Iterative}:
\begin{equation}\label{eqn17}
\left(\mathcal{R}^{k}\zeta\right)(x) = D(x)\zeta(x) + \sum_{j=1}^{K} \left( \int_{\Omega} m_{j}(x') \zeta(x') \, {\rm d}x' \right) n_j(x).
\end{equation}
The low-rank terms $n_{j} \otimes m_{j}$ are constructed by the iterative scheme \eqref{eqn11} on the fly.

Specifically, at the $k$th iteration, having obtained $u^{k+1}=\mathcal{P}_{[a,b]}(\eta^{k+1})$, we compute the state $y(u^{k+1})$ via \eqref{eqn1}.
Next we compute the auxiliary scattering field
\begin{equation*}
\hat{y}^{s}(u^{k+1}) \triangleq \mathcal{T}y_{\emptyset}(u^{k+1}) - \mathcal{T}y(u^{k+1}) = \mathcal{T}\mathcal{A}[y(u^{k+1})]^{-1}\mathcal{B}_{\tau}[y(u^{k+1})]u^{k+1} = \mathcal{H}[u^{k+1}]u^{k+1}.
\end{equation*}
Applying the adjoint HR-DtN operator yields the auxiliary dual function
\begin{equation*}
\hat{\zeta}^{k+1} = \mathcal{H}_{D}[u^{k+1}]^{*}\hat{y}^{s}(u^{k+1}) = \mathcal{H}_{D}[u^{k+1}]^{*}\mathcal{H}[u^{k+1}]u^{k+1} = \mathcal{G}_{D}[u^{k+1}]u^{k+1}.
\end{equation*}
The auxiliary pair $(u^{k+1}, \hat{\zeta}^{k+1})$ is then used to update the $\mathcal{R}^{k}$ by enforcing 
\begin{equation}\label{eqn15}
\mathcal{R}^{k+1}\hat{\zeta}^{k+1} = u^{k+1}.
\end{equation}
This condition ensures the exact reconstruction for the auxiliary problem.
\eqref{eqn15} is analogous to the secant condition in quasi-Newton methods.
Formally, $\mathcal{R}^{k+1}$ approximates the inverse of $\mathcal{G}_D[u^{k+1}]$, like a quasi-Newton matrix approximates the inverse Hessian.
This analogy motivates Broyden-type updates, e.g., Davidon–Fletcher–Powell (DFP) or Broyden–Fletcher–Goldfarb (BFG) formulae (cf. \cite[pp. 149]{Wright2006} and \cite[pp. 169]{Schnabel1996}), to construct the low-rank term in $\mathcal{R}$ iteratively.
It provides an efficient method to enforce \eqref{eqn15} while ensuring the symmetric positive definiteness of $\mathcal{R}$.

One critical limitation of existing resolvers in the IDSM is their inherent lack of stability.
Consider the delta-type operator $\eta(x) = D(x)\zeta(x)$ defined in \eqref{eqn13}.
Since the dual function $\zeta$ resides in $(L^{\infty}(\Omega))'$, which contains unbounded elements, the image $\eta$ generally falls outside $L^{\infty}(\Omega)$, and $\mathcal{R}$ is unbounded from $(L^{\infty}(\Omega))'$ to $L^{\infty}(\Omega)$.
The low-rank term also suffers from unboundedness in iterative schemes: as the iteration progresses, $\mathcal{R}^{k}$ approximates the unbounded inverse of the compact Gramian, and amplifies the errors uncontrollably.
The situation is exacerbated in the partial data setting, which usually necessitates more iterations, further aggravating the instability.
Moreover, \eqref{eqn15} leads to
\begin{equation*}
\mathcal{R}^{k+1}\hat{\zeta}^{k+1} = \mathcal{R}^{k+1}\mathcal{G}_{D}[u^{k+1}]u^{k+1} = u^{k+1} \quad \Rightarrow \quad \mathcal{R}^{k+1}\mathcal{G}_{D}[u^{k+1}] \approx \mathcal{I}_{L^{\infty}}.
\end{equation*}
However, the initial estimates $u^{k+1}$ are highly inaccurate, and $\mathcal{G}_D[u^{k+1}]$ poorly approximate $\mathcal{G}_D[u_{*}]$. Thus, the auxiliary pair $(u^{k+1}, \hat{\zeta}^{k+1})$ gives suboptimal corrections that degrade the convergence.
To address this issue, we propose a stabilization-correction scheme by employing a damping factor $\lambda_{k,p}$ to discount inaccurate auxiliary pairs.
It dynamically adjusts according to the accuracy of the auxiliary problem: higher-quality reconstructions result in a diminished damping factor, while lower-quality reconstructions lead to a larger damping factor.

For the singular part of the resolver, we define
\begin{equation}\label{eqn18}
\left(\mathcal{R}^{0}\zeta\right)(x) = \frac{D(x)^{1/2}}{|Q_{x}|^{1/2}}\left\langle \zeta(x'), \frac{D(x')^{1/2}}{|Q_{x'}|^{1/2}}\chi_{Q_{x}}(x') \right\rangle_{(L^{\infty}(\Omega))',L^{\infty}(\Omega)},
\end{equation}
where $D(x)$ is a bounded positive function, $Q_x \subset \Omega$ is a local neighborhood of $x$, and $|Q_x|$ and $\chi_{Q_x}$ denote its Lebesgue measure and characteristic function, respectively.
Similarly, we define the stabilizer $\mathcal{S} : (L^\infty(\Omega))' \to L^\infty(\Omega)$ as  
\begin{equation*}
\left(\mathcal{S}\zeta\right)(x) = \frac{1}{\sqrt{|Q_{x}|}}\int_{x'\in Q_{x}} \frac{1}{\sqrt{|Q_{x'}|}}\zeta(x') \, {\rm d}x'.
\end{equation*}
Let $h = \inf_{x \in \Omega} |Q_x|>0$.
Then $\mathcal{R}^0$ and $\mathcal{S}$ is bounded in the sense that  
\begin{equation}\label{eqn20}
\begin{aligned}
{}&\|\mathcal{R}^{0}\zeta\|_{L^{\infty}(\Omega)} \leq h^{-1}\|D\|_{L^{\infty}(\Omega)} \|\zeta\|_{L^{1}(\Omega)}, \quad \forall \zeta \in L^{1}(\Omega),\\
{}&\|\mathcal{S} \zeta\|_{L^\infty(\Omega)} \leq h^{-1} \|\zeta\|_{L^\infty(\Omega)}, \quad \forall \zeta \in L^{1}(\Omega),
\end{aligned}
\end{equation}
alongside the duality property of Banach spaces:
\begin{equation*}
\|f\|_{X'} = \sup_{x \in X \setminus \{0\}} \frac{\langle f, x \rangle_{X',X}}{\|x\|_{X}} = \sup_{x \in X'' \setminus \{0\}} \frac{\langle x, f \rangle_{X'',X'}}{\|x\|_{X''}}.
\end{equation*}
Moreover, due to the local support of $Q_x$, both $\mathcal{R}^0$ and $\mathcal{S}$ are efficient to compute.

Then we discuss the low-rank part of the resolver.
The notation $\langle\cdot,\cdot\rangle$ denotes the duality pairing between $(L^{\infty}(\Omega))'$ and $L^{\infty}(\Omega)$.
We replace condition \eqref{eqn15} with
\begin{equation*}
\mathcal{R}^{k+1}\hat{\zeta}^{k+1} = \hat{\eta}^{k+1} \quad \text{such that} \quad \mathcal{P}_{[a,b]}(\hat{\eta}^{k+1}) = u^{k+1},
\end{equation*}
where $\hat{\eta}^{k+1}$ denotes the auxiliary index function, whose explicit construction is given in \eqref{eqn24}.
At each iteration, $\mathcal{R}$ is updated through a stabilization-correction process.

For the stabilization stage, we introduce a stabilized resolver $\widetilde{\mathcal{R}}^{k}$ defined recursively by
\begin{equation}\label{eqn26}
\widetilde{\mathcal{R}}^{k} = \mathcal{R}^{0} + \frac{1}{1 + \lambda_{k-1,p}}\left(\widetilde{\mathcal{R}}^{k-1}-\mathcal{R}^{0}+\mathcal{S}\left(\mathcal{R}^{k}-\widetilde{\mathcal{R}}^{k-1}\right)\mathcal{S}\right),
\end{equation}
where $p \in [1, \infty]$ is an integrability index.
The critical component is the damping parameter $\lambda_{k-1,p} \geq 0$, which is computed from the auxiliary variables $\hat{\eta}^k$ and $\hat{\zeta}^k$ at iteration $k-1$.
For the DFP and BFG update schemes, it is given respectively by
\begin{equation}\label{eqn33}
\begin{aligned}
\lambda_{k-1,p}^{\rm DFP} &=
C_{\lambda}|\Omega|^{2/p^{*}}\bigg\| \frac{\hat{\eta}^k}{\sqrt{\langle \hat{\zeta}^k, \hat{\eta}^k \rangle}} + \frac{\widetilde{\mathcal{R}}^{k-1} \hat{\zeta}^k}{\sqrt{\langle \hat{\zeta}^k, \widetilde{\mathcal{R}}^{k-1} \hat{\zeta}^k \rangle}} \bigg\|_{L^{p}(\Omega)} \bigg\| \frac{\hat{\eta}^k}{\sqrt{\langle \hat{\zeta}^k, \hat{\eta}^k \rangle}} - \frac{\widetilde{\mathcal{R}}^{k-1} \hat{\zeta}^k}{\sqrt{\langle \hat{\zeta}^k, \widetilde{\mathcal{R}}^{k-1} \hat{\zeta}^k \rangle}} \bigg\|_{L^{p}(\Omega)},\\
\lambda_{k-1,p}^{\rm BFG} &=
C_{\lambda}|\Omega|^{2/p^{*}}\frac{\|\hat{\eta}^k\|_{L^{p}(\Omega)}}{\langle \hat{\zeta}^k, \hat{\eta}^k \rangle}  \bigg\| \left( 2 - \frac{\hat{\eta}^k \otimes \hat{\zeta}^k}{\langle \hat{\zeta}^k, \hat{\eta}^k \rangle} \right) \left( \hat{\eta}^k - \widetilde{\mathcal{R}}^{k-1} \hat{\zeta}^k \right) \bigg\|_{L^{p}(\Omega)},
\end{aligned}
\end{equation}
where $p^{*} = p/(p-1)$ is the H\"older conjugate of $p$ and $C_{\lambda}$ is a normalization constant.
We initialize with $\lambda_{-1,p} = 0$ and set $\lambda_{0,p} = 1$ by choosing $C_{\lambda}$ appropriately.
Since $\mathcal{S}:(L^\infty(\Omega))'\mapsto L^\infty(\Omega)\hookrightarrow (L^\infty(\Omega))'$ is bounded, $\widetilde{\mathcal{R}}^{k}$ is well-defined.

The damping factor $({1+\lambda_{k-1,p}})^{-1}$ is to adapt to the accuracy of $\widetilde{\mathcal{R}}^{k-1}$ for the auxiliary problem.
At the initial phase of the iteration, the error $\|\widetilde{\mathcal{R}}^{k-1}\hat{\zeta}^{k} - \hat{\eta}^{k}\|_{L^{p}(\Omega)}$ is large, which results in a large value of $\lambda_{k-1,p}$ and consequently a reduced damping factor.
When the error diminishes, i.e., $\|\widetilde{\mathcal{R}}^{k-1}\hat{\zeta}^{k} - \hat{\eta}^{k}\|_{L^{p}(\Omega)} \to 0^{+}$, the parameter $\lambda_{k-1,p}$ decreases, causing $(1+\lambda_{k-1,p})^{-1} \to 1^{-}$ and thereby preserving high-fidelity corrections.
Upon reaching convergence, $\lambda_{k-1,p}$ increases and stabilizes again due to the decay of the dual pairing $\langle \hat{\zeta}^{k}, \hat{\eta}^{k} \rangle$.
This behavior yields a distinct U-shaped trajectory: $\lambda_{k-1,p}$ begins high, reduces during the intermediate refinement stage, and finally increases and stabilizes.
The stabilization of $\lambda_{k-1,p}$ following this U-turn provides a practical indicator for convergence.
This U-shaped evolution is numerically verified in Fig. \ref{fig0}.

In the correction stage, a natural choice of the auxiliary index function is
\begin{equation}\label{eqn30}
\tilde{\eta}^{k+1}(x) = \begin{cases}
\max\left\{ b, \left( \widetilde{\mathcal{R}}^{k}\hat{\zeta}^{k+1} \right)(x) \right\}, & \text{if } u^{k+1}(x) = b, \\
u^{k+1}(x), & \text{if } u^{k+1}(x) \in (a,b), \\
\min\left\{ a, \left( \widetilde{\mathcal{R}}^{k}\hat{\zeta}^{k+1} \right)(x) \right\}, & \text{if } u^{k+1}(x) = a.
\end{cases}
\end{equation}
This construction ensures the exact reconstruction $\mathcal{P}_{[a,b]}(\tilde{\eta}^{k+1}) = u^{k+1}$ while simultaneously maintaining proximity to the value $\widetilde{\mathcal{R}}^{k}\hat{\zeta}^{k+1}$.
Consequently, enforcing $\mathcal{R}^{k+1}\hat{\zeta}^{k+1} = \tilde{\eta}^{k+1}$ necessitates the small adjustment to $\widetilde{\mathcal{R}}^{k}$, and limits the perturbation to the resolver between iterations.
Furthermore, it yields a reduced damping parameter $\lambda_{k,p}$ in the subsequent step, due to the small bias $\|\tilde{\eta}^{k+1}-\widetilde{\mathcal{R}}^{k}\hat{\zeta}^{k+1}\|_{L^{p}(\Omega)}$.

However, the stability of the low-rank update requires the strict positivity of the dual product $\langle \hat{\zeta}^{k+1}, \hat{\eta}^{k+1} \rangle$, but $\tilde{\eta}^{k+1}$ does not guarantee this property.
To enforce positivity, we introduce a final auxiliary index function $\hat{\eta}^{k+1}$ defined by the convex combination
\begin{equation}\label{eqn24}
\hat{\eta}^{k+1}(x) = \upsilon_{k+1}\tilde{\eta}^{k+1}(x) + (1-\upsilon_{k+1})u^{k+1}(x),
\end{equation}
where the parameter $\upsilon_{k+1} \in [0,1]$ is given by
\begin{equation}\label{eqn25}
\upsilon_{k+1} =\begin{cases}
\frac{\langle \hat{\zeta}^{k+1}, u^{k+1} \rangle}{2\left( \langle \hat{\zeta}^{k+1}, u^{k+1} \rangle - \langle \hat{\zeta}^{k+1}, \widetilde{\mathcal{R}}^{k}\hat{\zeta}^{k+1} \rangle \right)},&\text{ if }\langle\hat{\zeta}^{k+1},u^{k+1}\rangle>\langle \hat{\zeta}^{k+1},\widetilde{\mathcal{R}}^{k}\hat{\zeta}^{k+1}\rangle>\langle \hat{\zeta}^{k+1},\tilde{\eta}^{k+1}\rangle,\\
1,&\text{ else}.
\end{cases}
\end{equation}
As proven in Lemma \ref{lemma3}, this construction ensures $\langle \hat{\zeta}^{k+1}, \hat{\eta}^{k+1} \rangle > 0$.
We note that in practice, the condition for the first case in \eqref{eqn25} is seldom encountered, and $\upsilon_{k+1} = 1$ (i.e., $\hat{\eta}^{k+1} = \tilde{\eta}^{k+1}$) almost always suffices.
This mechanism thus serves primarily as a theoretical safeguard.

\begin{lemma}\label{lemma3}
The duality pairing $\langle \hat{\zeta}^{k+1}, \hat{\eta}^{k+1} \rangle$ is strictly positive if $\hat{\eta}^{k+1}$ is defined by \eqref{eqn24}.
\end{lemma}
\begin{proof}
The proof proceeds by induction.
The base case is established by the properties of the initial resolver $\mathcal{R}^0$.
For the inductive step, we assume the result holds for iteration $k$ and prove it for $k+1$.
The construction of $\hat{\eta}^{k+1}$ is designed to satisfy the projection condition $\mathcal{P}{[a,b]}(\hat{\eta}^{k+1}) = u^{k+1}$ while remaining close to $\widetilde{\mathcal{R}}^{k}\hat{\zeta}^{k+1}$.
We analyze all possible orderings of the three scalar values $\langle \hat{\zeta}^{k+1}, \tilde{\eta}^{k+1} \rangle$, $\langle \hat{\zeta}^{k+1}, \widetilde{\mathcal{R}}^{k}\hat{\zeta}^{k+1} \rangle$, and $\langle \hat{\zeta}^{k+1}, u^{k+1} \rangle$.
\begin{itemize}
\item[(a)] $\langle \hat{\zeta}^k, \tilde{\eta}^k \rangle \geq \langle \hat{\zeta}^k, \widetilde{\mathcal{R}}^{k-1} \hat{\zeta}^{k} \rangle$.
For $k > 0$, the inductive hypothesis ensures that $\widetilde{\mathcal{R}}^{k-1}$ preserves non-negativity, from which the desired inequality follows.
For $k = 0$, the non-negativity of $\hat{\zeta}^0$ is guaranteed by the construction of $\mathcal{R}^0$.
\item[(b)]
$\langle \hat{\zeta}^k, \tilde{\eta}^k \rangle < \langle \hat{\zeta}^k, \widetilde{\mathcal{R}}^{k-1} \hat{\zeta}^{k} \rangle$ and $\langle \hat{\zeta}^k, u^k \rangle \leq \langle \hat{\zeta}^k, \widetilde{\mathcal{R}}^{k-1} \hat{\zeta}^{k} \rangle$.
The projection condition implies $\langle \hat{\zeta}^k, \tilde{\eta}^k \rangle \geq \langle \hat{\zeta}^k, u^k \rangle$.
Since $\hat{\zeta}^k = \mathcal{G}[u^k]u^k$ and $\mathcal{G}[u^k]$ is positive definite, the desired result follows.
\item[(c)]
$\langle \hat{\zeta}^k, \tilde{\eta}^k \rangle < \langle \hat{\zeta}^k, \widetilde{\mathcal{R}}^{k-1} \hat{\zeta}^{k} \rangle < \langle \hat{\zeta}^k, u^k \rangle$.  
The projection property yields $\langle \hat{\zeta}^k, \tilde{\eta}^k \rangle \geq \langle \hat{\zeta}^k, 2u^k - \widetilde{\mathcal{R}}^{k-1} \hat{\zeta}^{k} \rangle$.
Combining this with the definition of $\upsilon_{k+1}$ gives the positivity.
\end{itemize}
These three cases exhaust all possibilities and together establish the lemma.
\end{proof}

Once $\hat{\eta}^{k+1}$ is obtained, the resolver $\mathcal{R}^{k+1}$ is updated via a low rank correction of the form $\mathcal{R}^{k+1} = \widetilde{\mathcal{R}}^{k} + \delta \mathcal{R}^{k}$, where $\delta \mathcal{R}^{k}$ for the DFP and BFG schemes are respectively given by
\begin{align}
\delta\mathcal{R}^{k}_{\mathrm{DFP}} =& \frac{\hat{\eta}^{k+1} \otimes \hat{\eta}^{k+1}}{\langle \hat{\zeta}^{k+1}, \hat{\eta}^{k+1} \rangle} - \frac{\widetilde{\mathcal{R}}^{k} \hat{\zeta}^{k+1} \otimes \widetilde{\mathcal{R}}^{k} \hat{\zeta}^{k+1}}{\langle \hat{\zeta}^{k+1}, \widetilde{\mathcal{R}}^{k} \hat{\zeta}^{k+1} \rangle}, \label{eqn:DFP} \\
\delta\mathcal{R}^{k}_{\mathrm{BFG}} =& \frac{(\hat{\eta}^{k+1} - \widetilde{\mathcal{R}}^{k} \hat{\zeta}^{k+1}) \otimes \hat{\eta}^{k+1} + \hat{\eta}^{k+1} \otimes (\hat{\eta}^{k+1} - \widetilde{\mathcal{R}}^{k} \hat{\zeta}^{k+1})}{\langle \hat{\zeta}^{k+1}, \hat{\eta}^{k+1} \rangle} \nonumber \\
{}& - \frac{\langle \hat{\eta}^{k+1} - \widetilde{\mathcal{R}}^{k} \hat{\zeta}^{k+1}, \hat{\zeta}^{k+1} \rangle}{\langle \hat{\zeta}^{k+1}, \hat{\eta}^{k+1} \rangle^{2}} \hat{\eta}^{k+1} \otimes \hat{\eta}^{k+1}.\label{eqn:BFG}
\end{align}

This update strategy ensures the uniform spectral boundedness of $\widetilde{\mathcal{R}}^{k}$:
\begin{theorem}\label{thm1}
The following inequality holds for all $k \in \mathbb{N}$ and $\zeta \in (L^{\infty}(\Omega))'$:
\begin{equation}\label{eqn19}
\langle \zeta, \widetilde{\mathcal{R}}^{k}\zeta\rangle \leq (h^{-1}\|D\|_{L^\infty(\Omega)}+(C_{\lambda}h^{2})^{-1} )\|\zeta\|_{(L^\infty(\Omega))'}^2.
\end{equation}
\end{theorem}
\begin{proof}
We establish the inequality
\begin{equation}\label{eqn22}
\langle \xi, (\widetilde{\mathcal{R}}^{k}-\mathcal{R}^{0})\xi\rangle \leq (C_{\lambda}h^{2})^{-1}\|\xi\|_{(L^{\infty}(\Omega))'}^{2},
\end{equation}
by mathematical induction. The theorem then follows directly from the assumption \eqref{eqn20}.
For the case $k=0$, \eqref{eqn22} holds trivially.
Now assume that the assertion holds for some $k$, and we prove the stability for $\widetilde{\mathcal{R}}^{k+1}$.
For the DFP correction, we deduce that for any $\xi \in (L^{\infty}(\Omega))'$:
\begin{align*}
{}&\langle \xi, \mathcal{S}(\mathcal{R}^{k+1} - \widetilde{\mathcal{R}}^{k})\mathcal{S}\xi \rangle \\
=& \frac{\langle \mathcal{S}\xi, \hat{\eta}^{k+1} \rangle^{2}}{\langle \hat{\zeta}^{k+1}, \hat{\eta}^{k+1} \rangle} - \frac{\langle \mathcal{S}\xi, \widetilde{\mathcal{R}}^{k}\hat{\zeta}^{k+1} \rangle^{2}}{\langle \hat{\zeta}^{k+1}, \widetilde{\mathcal{R}}^{k}\hat{\zeta}^{k+1} \rangle}
= \left\langle \mathcal{S}\xi, \frac{\hat{\eta}^{k+1}}{\sqrt{\langle \hat{\zeta}^{k+1}, \hat{\eta}^{k+1} \rangle}} \right\rangle^{2} - \left\langle \mathcal{S}\xi, \frac{\widetilde{\mathcal{R}}^{k}\hat{\zeta}^{k+1}}{\sqrt{\langle \hat{\zeta}^{k+1}, \widetilde{\mathcal{R}}^{k}\hat{\zeta}^{k+1} \rangle}} \right\rangle^{2} \\
=& \left\langle \mathcal{S}\xi, \frac{\hat{\eta}^{k+1}}{\sqrt{\langle \hat{\zeta}^{k+1}, \hat{\eta}^{k+1} \rangle}} - \frac{\widetilde{\mathcal{R}}^{k}\hat{\zeta}^{k+1}}{\sqrt{\langle \hat{\zeta}^{k+1}, \widetilde{\mathcal{R}}^{k}\hat{\zeta}^{k+1} \rangle}} \right\rangle
{} \times \left\langle \mathcal{S}\xi, \frac{\hat{\eta}^{k+1}}{\sqrt{\langle \hat{\zeta}^{k+1}, \hat{\eta}^{k+1} \rangle}} + \frac{\widetilde{\mathcal{R}}^{k}\hat{\zeta}^{k+1}}{\sqrt{\langle \hat{\zeta}^{k+1}, \widetilde{\mathcal{R}}^{k}\hat{\zeta}^{k+1} \rangle}} \right\rangle 
\end{align*}
Upon letting
$f_{1}= \frac{\hat{\eta}^{k+1}}{\sqrt{\langle \hat{\zeta}^{k+1}, \hat{\eta}^{k+1} \rangle}} - \frac{\widetilde{\mathcal{R}}^{k}\hat{\zeta}^{k+1}}{\sqrt{\langle \hat{\zeta}^{k+1} \widetilde{\mathcal{R}}^{k}\hat{\zeta}^{k+1} \rangle}}$ and $
f_{2}=\frac{\hat{\eta}^{k+1}}{\sqrt{\langle \hat{\zeta}^{k+1}, \hat{\eta}^{k+1} \rangle}} + \frac{\widetilde{\mathcal{R}}^{k}\hat{\zeta}^{k+1}}{\sqrt{\langle \hat{\zeta}^{k+1}, \widetilde{\mathcal{R}}^{k}\hat{\zeta}^{k+1} \rangle}}$,
we deduce 
\begin{align*}
\langle \xi, \mathcal{S}(\mathcal{R}^{k+1} - \widetilde{\mathcal{R}}^{k})\mathcal{S}\xi \rangle=&\left\langle \mathcal{S}\xi,f_{1} \right\rangle\times \left\langle \mathcal{S}\xi, f_{2}\right\rangle\leq \|\mathcal{S}\xi\|_{L^{p^{*}}(\Omega)}^{2} \left\|f_{1} \right\|_{L^{p}(\Omega)}
\left\|f_{2} \right\|_{L^{p}(\Omega)} \\
=& \frac{\lambda_{k,p}}{C_{\lambda}|\Omega|^{2/p^{*}}}\|\mathcal{S}\xi\|_{L^{p^{*}}(\Omega)}^{2} \leq \frac{\lambda_{k,p}}{C_{\lambda}}\|\mathcal{S}\xi\|_{L^{\infty}(\Omega)}^{2} \leq\frac{\lambda_{k,p}}{C_{\lambda}h^{2}}\|\xi\|_{(L^{\infty}(\Omega))'}^{2},
\end{align*}
For the BFG correction, similarly, we get
\begin{align*}
\langle \xi, \mathcal{S}(\mathcal{R}^{k+1} - \widetilde{\mathcal{R}}^{k})\mathcal{S}\xi \rangle
=& 2 \frac{\langle \mathcal{S}\xi, \hat{\eta}^{k+1} \rangle \langle \xi, \hat{\eta}^{k+1} - \widetilde{\mathcal{R}}^{k}\hat{\zeta}^{k+1} \rangle}{\langle \hat{\zeta}^{k+1}, \hat{\eta}^{k+1} \rangle} - \frac{\langle \mathcal{S}\xi, \hat{\eta}^{k+1} \rangle^{2} \langle \hat{\zeta}^{k+1}, \hat{\eta}^{k+1} - \widetilde{\mathcal{R}}^{k}\hat{\zeta}^{k+1} \rangle}{\langle \hat{\zeta}^{k+1}, \hat{\eta}^{k+1} \rangle^{2}} \\
=& \frac{\langle \mathcal{S}\xi, \hat{\eta}^{k+1} \rangle}{\langle \hat{\zeta}^{k+1}, \hat{\eta}^{k+1} \rangle} \left\langle \mathcal{S}\xi, \left(2 - \frac{\hat{\eta}^{k+1} \otimes \hat{\zeta}^{k+1}}{\langle \hat{\zeta}^{k+1}, \hat{\eta}^{k+1} \rangle}\right)\left(\hat{\eta}^{k+1} - \widetilde{\mathcal{R}}^{k}\hat{\zeta}^{k+1}\right) \right\rangle \\
\leq& \|\mathcal{S}\xi\|_{L^{p^{*}}(\Omega)}^{2} \frac{\|\hat{\eta}^{k+1}\|_{L^{p}(\Omega)}}{\langle \hat{\zeta}^{k+1}, \hat{\eta}^{k+1} \rangle} \left\| \left(2 - \frac{\hat{\eta}^{k+1} \otimes \hat{\zeta}^{k+1}}{\langle \hat{\zeta}^{k+1}, \hat{\eta}^{k+1} \rangle}\right)\left(\hat{\eta}^{k+1} - \widetilde{\mathcal{R}}^{k}\hat{\zeta}^{k+1}\right) \right\|_{L^{p}(\Omega)} \\
=& \frac{\lambda_{k,p}}{C_{\lambda}|\Omega|^{2/p^{*}}}\|\mathcal{S}\xi\|_{L^{p^{*}}(\Omega)}^{2} \leq \frac{\lambda_{k,p}}{C_{\lambda}}\|\mathcal{S}\xi\|_{L^{\infty}(\Omega)}^{2} \leq\frac{\lambda_{k,p}}{C_{\lambda}h^{2}}\|\xi\|_{(L^{\infty}(\Omega))'}^{2}.
\end{align*}
Combining these result gives 
$$\langle \xi, \mathcal{S}(\mathcal{R}^{k+1} - \widetilde{\mathcal{R}}^{k})\mathcal{S}\xi \rangle \leq \lambda_{k,p}\left(C_{\lambda}h^{2}\right)^{-1}\|\xi\|_{(L^{\infty}(\Omega))'}^{2},$$ 
which implies
\begin{equation*}
\begin{aligned}
&\langle \xi, (\widetilde{\mathcal{R}}^{k+1}-\mathcal{R}^{0})\xi \rangle = \left\langle \xi,(1 + \lambda_{k,p})^{-1}\left(\widetilde{\mathcal{R}}^{k}-\mathcal{R}^{0}+\mathcal{S}\left(\mathcal{R}^{k+1}-\widetilde{\mathcal{R}}^{k}\right)\mathcal{S}\right) \xi \right\rangle\\
=& \frac{\langle \xi, (\widetilde{\mathcal{R}}^{k}-\mathcal{R}^{0} ) \xi \rangle+\langle \xi, \mathcal{S}(\mathcal{R}^{k+1}- \widetilde{\mathcal{R}}^{k})\mathcal{S} \xi \rangle}{1 + \lambda_{k,p}}
\leq \left(C_{\lambda}h^{2}\right)^{-1}\|\xi\|_{(L^{\infty}(\Omega))'}^{2}.
\end{aligned}
\end{equation*}
This completes the inductive step and proves that \eqref{eqn22} holds for all iterations $k$.
\end{proof}

Theorem \ref{thm1} gives the stability for the IDSM, i.e., the {uniform} spectral boundedness of the stabilized operator $\widetilde{\mathcal{R}}^{k}$.
In $\mathcal{K}^k = \mathcal{R}^k \mathcal{H}_D[u^k]^*$, the adjoint $\mathcal{H}_D[u^k]^*$ can be obtained via two well-posed elliptic problems, cf.  Section \ref{sec:IMP}.
The stability of the IDSM relies crucially on the boundedness of $\mathcal{R}^k$.
The boundedness of $\widetilde{\mathcal{R}}^k$ ensures uniform stability of $\mathcal{R}^{k}$.
In contrast, $\mathcal{R}^k$ is not necessarily so, since  $\mathcal{R}$ approximates the unbounded inverse of the compact Gramian $\mathcal{G}_D[u_*]$.
The stabilization-correction scheme balances the stability and fidelity: the stabilization step enforces uniform boundedness of $\widetilde{\mathcal{R}}^k$, and the correction step refines $\mathcal{R}^k$ to improve the accuracy.

In the proof of Theorem \ref{thm1}, the spectral boundedness of $\widetilde{\mathcal{R}}^k$ is established in the $L^{\infty}(\Omega)$ norm via the inequality $\frac{1}{|\Omega|^{2/p^{*}}}\|\mathcal{S}\xi\|_{L^{p^{*}}(\Omega)}^{2} \leq \|\mathcal{S}\xi\|_{L^{\infty}(\Omega)}^{2}$, which holds for any $p \in [1, \infty)$.
The analysis is valid for any $L^p$ spaces.
The term $\frac{1}{|\Omega|^{2/p^{*}}}\|\mathcal{S}\xi\|_{L^{p^{*}}(\Omega)}^{2}$ is monotone increasing with respect to $p^{*}$ (and hence decreasing with respect to $p$), and a larger $p$ value yields a tighter upper bound.
(i.e., a smaller spectral bound for $\widetilde{\mathcal{R}}^k$), and the damping factor $\frac{1}{1+\lambda_{k,p}}$ in \eqref{eqn26}.
This behavior reflects the role of the damping factor in balancing stability and fidelity: a larger $p$ leads to stronger damping (larger $\lambda_{p,k}$), and suppresses low-rank corrections more aggressively to preserve stability, while a smaller $p$ imposes weaker damping (smaller $\lambda_{p,k}$), retaining more correction terms for accuracy.
Nevertheless, the reconstructions is insensitive to the choice of $p$.

\section{Algorithmic implementation}\label{sec:IMP}
This section details the numerical implementation of the proposed IDSM.
We begin by deriving a computable expression for the adjoint operator $\mathcal{H}_D[u^k]^*$ via two elliptic problems in Lemma \ref{lemma2}.
We then describe the practical realization of the stabilization-correction scheme.
Finally, we consolidate these components into a complete algorithmic description in Algorithm \ref{alg1}.

\begin{lemma}\label{lemma2}
For any $v\in L^{2}(\Gamma)$, let $(w_{1}, p_{1}) \in H^{1}(\Omega) \times L^{2}(\Gamma)$ solve
\begin{equation}\label{eqn29}
\left\{
\begin{aligned}
\langle \mathcal{A}[y(u^{k})] w_{1}, z \rangle_{(H^{1}(\Omega))', H^{1}(\Omega)} - (p_{1}, \mathcal{T} z)_{L^{2}(\Gamma)} &= 0, \quad &&\forall z \in H^{1}(\Omega), \\
(\mathcal{T} w_{1}, q)_{L^{2}(\Gamma)} + (\alpha_{D} p_{1}, q)_{L^{2}(\Gamma)} &= (v, q)_{L^{2}(\Gamma)}, \quad &&\forall q \in L^{2}(\Gamma),
\end{aligned}
\right.
\end{equation}
and let $(w_{2}, p_{2}) \in H^{1}(\Omega) \times L^{2}(\Gamma)$ solve
\begin{equation}\label{eqn27}
\left\{
\begin{aligned}
\langle \mathcal{A}[y(u^{k})]^{*} w_{2}, z \rangle_{(H^{1}(\Omega))', H^{1}(\Omega)} + (p_{2}, \mathcal{T} z)_{L^{2}(\Gamma)} &= 0, \quad &&\forall z \in H^{1}(\Omega), \\
-(\mathcal{T} w_{2}, q)_{L^{2}(\Gamma)} + (\alpha_{D} p_{2}, q)_{L^{2}(\Gamma)} &= (p_{1}, q)_{L^{2}(\Gamma)}, \quad &&\forall q \in L^{2}(\Gamma).
\end{aligned}
\right.
\end{equation}
Then we have $\mathcal{H}_{D}[u^{k}]^{*} v=-\mathcal{B}_{\tau}[y(u^{k})]^{*} w_{2}$.
\end{lemma}
\begin{proof}
First we verify $p_2 = \Lambda_{\alpha,D}(\mathcal{A}[y(u^{k})])^{*} p_1$.
For any $\tilde{v} \in L^{2}(\Gamma)$, let $(\tilde{w}_{2}, \tilde{p}_{2})$ solve 
\begin{equation}\label{eqn28}
\left\{
\begin{aligned}
\langle \mathcal{A}[y(u^{k})] \tilde{w}_{2}, z \rangle_{(H^{1}(\Omega))', H^{1}(\Omega)} - (\tilde{p}_{2}, \mathcal{T} z)_{L^{2}(\Gamma)} &= 0, \quad &\forall z \in H^{1}(\Omega), \\
(\mathcal{T} \tilde{w}_{2}, q)_{L^{2}(\Gamma)} + (\alpha_{D} \tilde{p}_{2}, q)_{L^{2}(\Gamma)} &= (\tilde{v}, q)_{L^{2}(\Gamma)}, \quad &\forall q \in L^{2}(\Gamma).
\end{aligned}
\right.
\end{equation}
The definition \eqref{eqn12} of the HR-DtN map implies 
\begin{equation*} 
\tilde{p}_{2} = \Lambda{\alpha,D}(\mathcal{A}[y(u^{k})]) \tilde{v}\quad \mbox{and}\quad p_{1} = \Lambda_{\alpha,D}(\mathcal{A}[y(u^{k})]) v.
\end{equation*}
By the second line of \eqref{eqn27}) and the first line of \eqref{eqn28}, we deduce
\begin{align*}
& (\Lambda_{\alpha,D}(\mathcal{A}[y(u^{k})])^{*}p_{1},  \tilde{v})_{L^{2}(\Gamma)}=(p_{1}, \Lambda_{\alpha,D}(\mathcal{A}[y(u^{k})]) \tilde{v})_{L^{2}(\Gamma)} = (p_{1}, \tilde{p}_{2})_{L^{2}(\Gamma)} \\
{=}& -(\mathcal{T} w_{2}, \tilde{p}_{2})_{L^{2}(\Gamma)} + (\alpha_{D} p_{2}, \tilde{p}_{2})_{L^{2}(\Gamma)} 
{=} -\langle \mathcal{A}[y(u^{k})] \tilde{w}_{2}, w_{2} \rangle_{(H^{1}(\Omega))', H^{1}(\Omega)} + (\alpha_{D} \tilde{p}_{2}, p_{2})_{L^{2}(\Gamma)} \\
=& -\langle \mathcal{A}[y(u^{k})]^{*} w_{2}, \tilde{w}_{2} \rangle_{(H^{1}(\Omega))', H^{1}(\Omega)} + (\alpha_{D} \tilde{p}_{2}, p_{2})_{L^{2}(\Gamma)}.
\end{align*}
Now by the first line of \eqref{eqn27} and the second line of \eqref{eqn28}, we deduce
\begin{align*}
(\Lambda_{\alpha,D}(\mathcal{A}[y(u^{k})])^{*}p_{1},  \tilde{v})_{L^{2}(\Gamma)}&{=} (\mathcal{T} w_{2}, \tilde{p}_{2})_{L^{2}(\Gamma)} + (\alpha_{D} \tilde{p}_{2}, p_{2})_{L^{2}(\Gamma)}
{=}(p_{2}, \tilde{v})_{L^{2}(\Gamma)},
\end{align*}
Since $\tilde{v}$ is arbitrary, we get $p_2 = \Lambda_{\alpha,D}(\mathcal{A}[y(u^{k})])^{*} p_1$. The first line in \eqref{eqn27} yields 
$$w_{2} = -\mathcal{A}[y(u^{k})]^{-*} \mathcal{T}^{*} p_{2}.$$
For any $q \in L^{\infty}(\Omega)$, the definitions of $\mathcal{H}[u]$, and $\mathcal{H}_{D}[u]$ imply
\begin{align*}
&\langle \mathcal{B}_{\tau}[y(u^{k})]^{*} w_{2}, q \rangle_{(L^{\infty}(\Omega))', L^{\infty}(\Omega)} = \langle -\mathcal{B}_{\tau}[y(u^{k})]^{*} \mathcal{A}[y(u^{k})]^{-*} \mathcal{T}^{*} p_{2}, q \rangle_{(L^{\infty}(\Omega))', L^{\infty}(\Omega)} \\
{=}& -\langle \mathcal{H}[u^{k}]^{*} p_{2}, q \rangle_{(L^{\infty}(\Omega))', L^{\infty}(\Omega)} 
= -(\Lambda_{\alpha,D}(\mathcal{A}[y(u^{k})])^{*} p_{1}, \mathcal{H}[u^{k}] q)_{L^{2}(\Gamma)}\\
=& -(\Lambda_{\alpha,D}(\mathcal{A}[y(u^{k})]) v, \Lambda_{\alpha,D}(\mathcal{A}[y(u^{k})]) \mathcal{H}[u^{k}] q)_{L^{2}(\Gamma)}{=}-\langle \mathcal{H}_{D}[u^{k}]^{*} v, q \rangle_{(L^{\infty}(\Omega))', L^{\infty}(\Omega)},
\end{align*}
This shows $\mathcal{B}_{\tau}[y(u^{k})]^{*} w_{2} = -\mathcal{H}_{D}[u^{k}]^{*} v$, and completes the proof of the lemma.
\end{proof}

The adjoint $\mathcal{B}_{\tau}[y(u^{k})]^{*}$ accounts for different inclusion types and aggregates multiple datasets.
Consider the scenario with $L$ types of inclusions with the corresponding operators $\{\mathcal{B}^{\ell}\}_{\ell=1}^L$, and $I$ measurements for sources $\{f_i\}_{i=1}^I$.
Then the operator $\mathcal{B}$ and parameter $u$ are vector-valued, with
\begin{equation*}
\mathcal{B}[u](y)=(\mathcal{B}^{1}[\cdot](y),\mathcal{B}^{2}[\cdot](y),\dots,\mathcal{B}^{L}[\cdot](y))(u^{1},u^{2},\dots,u^{L})^{\top}=\sum_{\ell=1}^{L} \mathcal{B}^{\ell}[u_{\ell}](y),
\end{equation*}
where each $u_{\ell}$ takes the form \eqref{eqn2}.
The governing system is given by
\begin{equation*}
\mathcal{A}[y_i]y_i + \sum_{\ell=1}^{L} \mathcal{B}^{\ell}[u_{\ell}](y_i) = f_i, \quad i = 1, \dots, I.
\end{equation*}
Then the scattering fields across all $I$ datasets can be represented in a matrix-vector form:
\begin{equation*}
\begin{pmatrix} y^{s}_{1} \\ \vdots \\ y^{s}_{I} \end{pmatrix} = \mathcal{T} \begin{pmatrix}
\mathcal{A}[y_{1}]^{-1}\mathcal{B}^{1}_{\tau}[y_{1}] & \mathcal{A}[y_{1}]^{-1}\mathcal{B}^{2}_{\tau}[y_{1}] & \cdots & \mathcal{A}[y_{1}]^{-1}\mathcal{B}^{L}_{\tau}[y_{1}] \\
\vdots & \vdots & \ddots & \vdots \\
\mathcal{A}[y_{I}]^{-1}\mathcal{B}^{1}_{\tau}[y_{I}] & \mathcal{A}[y_{I}]^{-1}\mathcal{B}^{2}_{\tau}[y_{I}] & \cdots & \mathcal{A}[y_{I}]^{-1}\mathcal{B}^{L}_{\tau}[y_{I}]
\end{pmatrix} \begin{pmatrix} u_{1} \\ u_{2} \\ \vdots \\ u_{L} \end{pmatrix},
\end{equation*}
where $\mathcal{T}$ acts component-wise on the vector.
Thus, the forward operator $\mathcal{H}[u]$ is an operator matrix mapping from $(L^{\infty}(\Omega))^{L}$, i.e., $(u_{1},u_{2},\dots,u_{L})^{\top}$, to $(L^{2}(\Gamma))^{I}$, i.e., $(y^{s}_{1},\dots,y^{s}_{I})^{\top}$.
The adjoint operator $\mathcal{H}[u]^{*}$ maps from $(L^{2}(\Gamma))^{I}$ to $((L^{\infty}(\Omega))')^{L}$, with the dual function $\zeta$ given by
\begin{align*}
\zeta =& \begin{pmatrix} \zeta_{1} \\ \zeta_{2} \\ \vdots \\ \zeta_{L} \end{pmatrix}
= \begin{pmatrix}
\mathcal{B}^{1}_{\tau}[y_{1}]^{*} & \cdots & \mathcal{B}^{1}_{\tau}[y_{I}]^{*} \\
\mathcal{B}^{2}_{\tau}[y_{1}]^{*} & \cdots & \mathcal{B}^{2}_{\tau}[y_{I}]^{*} \\
\vdots & \ddots & \vdots \\
\mathcal{B}^{L}_{\tau}[y_{1}]^{*} & \cdots & \mathcal{B}^{L}_{\tau}[y_{I}]^{*}
\end{pmatrix} \begin{pmatrix} \mathcal{A}[y_{1}]^{-*}\mathcal{T}^{*}y^{s}_{d,1} \\ \vdots \\ \mathcal{A}[y_{I}]^{-*}\mathcal{T}^{*}y^{s}_{d,I} \end{pmatrix}.
\end{align*}
This allows distinguishing inclusions from different physical nature. For the $\ell$th inclusion  type, the index function is given by
\begin{equation*}
\zeta^{\ell} = \sum_{i=1}^{I} \mathcal{B}_{\tau}^{\ell}[y_{i}] \mathcal{A}[y_{i}]^{-*} \mathcal{T}^{*} y^{s}_{d,i},
\end{equation*}
which aggregates information across all scattering fields.

We now detail the numerical construction of the resolver.
The operators $\mathcal{R}^0$ and $\mathcal{S}$ are applied with a mesh coarsening technique.
The domain is discretized via the finite element method with two meshes: a fine mesh $\mathscr{T}_f$ and a coarse mesh $\mathscr{T}_c$.
For a piecewise linear function $\zeta$ defined on $\mathscr{T}_f$, the action of $\mathcal{R}^0$ is computed in three sequential steps: first, $\zeta$ is multiplied pointwise by $\sqrt{D(x)}$; second, the resulting function is projected onto a piecewise constant basis over the coarse mesh $\mathscr{T}_c$; third, the projected values are multiplied again by $\sqrt{D(x)}$.
The stabilizer $\mathcal{S}$ is implemented in a similar manner.
At the $k$th iteration, at the stabilization stage, we apply the operator $\mathcal{S}$ to the low-rank term $\delta \mathcal{R}^{k-1}$, cf. \eqref{eqn:DFP} or \eqref{eqn:BFG}, by projecting $\hat{\eta}^{k}$ and $\widetilde{\mathcal{R}}^{k-1}\hat{\zeta}^{k}$ onto the coarse mesh $\mathscr{T}_{c}$:
\begin{equation*}
\begin{aligned}
\mathcal{S}\delta\mathcal{R}^{k-1}_{\mathrm{DFP}}\mathcal{S} &= \frac{\mathcal{S}\hat{\eta}^{k} \otimes \mathcal{S}\hat{\eta}^{k}}{\langle \hat{\zeta}^{k}, \hat{\eta}^{k} \rangle_{(L^{\infty}(\Omega))', L^{\infty}(\Omega)}} - \frac{\mathcal{S}\widetilde{\mathcal{R}}^{k-1} \hat{\zeta}^{k} \otimes \mathcal{S}\widetilde{\mathcal{R}}^{k-1} \hat{\zeta}^{k}}{\langle \hat{\zeta}^{k}, \widetilde{\mathcal{R}}^{k-1} \hat{\zeta}^{k} \rangle_{(L^{\infty}(\Omega))', L^{\infty}(\Omega)}}, \\
\mathcal{S}\delta\mathcal{R}_{\mathrm{BFG}}^{k-1}\mathcal{S} &= \frac{\mathcal{S}(\hat{\eta}^{k} - \widetilde{\mathcal{R}}^{k-1} \hat{\zeta}^{k}) \otimes \mathcal{S}\hat{\eta}^{k} + \mathcal{S}\hat{\eta}^{k} \otimes \mathcal{S}(\hat{\eta}^{k} - \widetilde{\mathcal{R}}^{k-1} \hat{\zeta}^{k})}{\langle \hat{\zeta}^{k}, \hat{\eta}^{k} \rangle_{(L^{\infty}(\Omega))', L^{\infty}(\Omega)}} \\
&\quad - \frac{\langle \hat{\eta}^{k} - \widetilde{\mathcal{R}}^{k-1} \hat{\zeta}^{k}, \hat{\zeta}^{k} \rangle_{L^{\infty}(\Omega), L^{1}(\Omega)}}{\langle \hat{\zeta}^{k}, \hat{\eta}^{k} \rangle_{(L^{\infty}(\Omega))', L^{\infty}(\Omega)}^{2}} \mathcal{S}\hat{\eta}^{k} \otimes \mathcal{S}\hat{\eta}^{k}.
\end{aligned}
\end{equation*}
Then all low-rank terms is damped by the factor $\frac{1}{1+\lambda_{k-1,p}}$.
In the correction stage, the following quantities are computed and stored: the auxiliary index function $\hat{\eta}^{k+1}$; function $\widetilde{\mathcal{R}}^{k}\hat{\zeta}^{k+1}$; the dual pairings $\langle \hat{\zeta}^{k+1},\hat{\eta}^{k+1}\rangle_{(L^{\infty}(\Omega))',L^{\infty}(\Omega)}$ and $\langle \hat{\zeta}^{k+1},\widetilde{\mathcal{R}}^{k}\hat{\zeta}^{k+1}\rangle_{(L^{\infty}(\Omega))',L^{\infty}(\Omega)}$; and the damping parameter $\lambda_{k,p}$.
These computed values are subsequently employed in the low-rank update formulae \eqref{eqn:DFP} or \eqref{eqn:BFG} to construct the resolver $\mathcal{R}^{k+1}$ for the subsequent iteration.
The damping parameter $\lambda_{k,p}$ is preserved for use in the stabilization stage of the next iterative step.

The detailed algorithm for the IDSM is given in in Algorithm \ref{alg1}.

\begin{algorithm}[hbt!]
\caption{IDSM for Elliptic Inverse Problems with Partial Cauchy Data}
\label{alg1}
\begin{algorithmic}[1]
\Require{Boundary measurements $y_d$, diagonal function $D(x)$, regularization parameters $(\alpha_{d}, \alpha_{n})$, integral index $p$, maximum iterations $K$}
\Ensure{Reconstructed inclusion $u^K$ and state $y(u^K)$}

\State \textbf{Initialization:}
\State Set initial guess $u^0 \gets 0$ and $\lambda_{-1,p} \gets 0$.
\State Solve forward problem for $y(u^0)$ via \eqref{eqn1}.
\State Set background solution $y_\emptyset(u^0)$ via \eqref{eqn6}.

\For{$k = 0, 1, 2, \dots, K-1$}
\State \textbf{Data completion:}
\State $\tilde{y}_d(u^k) \gets y_d + \mathcal{T}_N y(u^k)$.

\State \textbf{Dual function computation:}
\State Compute $\zeta^k \gets \mathcal{H}_D[u^k]^* (y_\emptyset(u^k)-\tilde{y}_d(u^k))$ via Lemma~\ref{lemma2}.

\State \textbf{Index function computation:}
\State $\eta^{k+1} \gets \mathcal{R}^k \zeta^k$.

\State \textbf{Parameter update:}
\State $u^{k+1} \gets \mathcal{P}_{[a,b]}(\eta^{k+1})$.
\State Solve $y(u^{k+1})$ via \eqref{eqn1}.

\If{$k \neq K-1$}
\State \textbf{Auxiliary field computation:}
\State Solve $y_\emptyset(u^{k+1})$ via \eqref{eqn6}. \Comment{For nonlinear $\mathcal{A}$ only}
\State $\hat{y}^s_d(u^{k+1}) \gets y_\emptyset(u^{k+1}) - \mathcal{T} y(u^{k+1})$.
\State Compute $\hat{\zeta}^{k+1} \gets \mathcal{H}_D[u^{k+1}]^* \hat{y}^s_d(u^{k+1})$ via Lemma~\ref{lemma2}.

\State \textbf{Resolver stabilization:}
\State Project latest low-rank terms onto coarse mesh $\mathscr{T}_c$.
\If{$k=0$}
\State Set $C_{\lambda}$ such that $\lambda_{0,p}=1.0$.
\EndIf
\State Apply damping factor $\frac{1}{1+\lambda_{k-1,p}}$ to low-rank terms.

\State \textbf{Auxiliary index computation:}
\State Compute $\widetilde{\mathcal{R}}^k \hat{\zeta}^{k+1}$ and $\hat{\eta}^{k+1}$ via \eqref{eqn30}.

\State \textbf{Resolver update:}
\State Update $\mathcal{R}^{k+1} \gets \widetilde{\mathcal{R}}^k + \delta \mathcal{R}^{k}$ using \eqref{eqn:DFP} or \eqref{eqn:BFG}.
\State Update scaling constant $C_D$ via \eqref{eqn16}.

\State \textbf{Damping factor update:}
\State Compute $\lambda_{k,p}$ for next iteration.
\EndIf
\EndFor

\State \textbf{Return} $u^K$, $y(u^K)$
\end{algorithmic}
\end{algorithm}

Finally, we discuss the complexity of Algorithm \ref{alg1} with one pair of Cauchy data.
Solving elliptic equations (\eqref{eqn1}, \eqref{eqn6}, and \eqref{eqn29}–\eqref{eqn27}) represents the dominant cost.
Other operations, e.g., the adjoint $\mathcal{B}_{\tau}[y]^{*}$, resolver $\mathcal{R}^{k}$, or stabilized resolver $\widetilde{\mathcal{R}}^{k}$, are cheap.
Thus, we count the complexity of the algorithm by the number of elliptic PDE solves:
(i) The initialization phase (steps 1–3) requires one single PDE solve \eqref{eqn6};
(ii) Steps 5–8 involve solving problem \eqref{eqn1} and \eqref{eqn29}–\eqref{eqn27}, $3$ PDE solves.
(iii) Steps 10–15 require one adjoint solve and one background solve \eqref{eqn6} (for nonlinear $\mathcal{A}$).
For linear $\mathcal{A}$, each iteration requires $2$ PDE solves; for nonlinear $\mathcal{A}$, $3$ solves.
Thus, after $K$ iterations, totals are $5K-1$ (linear) or $6K-2$ (nonlinear) solves.
Numerical experiments in Section \ref{sec:NUM} {show that the IDSM can give quite reasonable reconstructions within about one dozen iterations, requiring only 30–70 elliptic PDE solves.}

\section{Numerical experiments} \label{sec:NUM}
\newcommand{\sevenfiglen}{0.11\textwidth}

Now we present numerical results to illustrate the proposed IDSM.
Example \ref{exam1} shows the noise robustness of the IDSM.
Example \ref{exam2} examines the DOT problem in \eqref{eqn3} with diverse $\Gamma_D$.
Example \ref{exam3} revisits the DOT model with various integral index $p$, and Example \ref{exam4} studies the cardiac electrophysiology (CE) model \cite{Scacchi2014}.
Finally, Example \ref{exam5} recovers a modulus nonlinear potential inhomogeneity.
These experiments show the stability and effectiveness of the IDSM.
In Section \ref{subsec:pltLambda}, we analyze the damping factor $\frac{1}{\lambda_{k-1,p}}$.

The numerical experiments are conducted on the unit disk $\Omega = \{(x_1, x_2) \mid x_1^2 + x_2^2 < 1\}$.
The exact data $y(u_*)$ is computed by solving problem \eqref{eqn1} on a fine mesh with 42996 triangles.
The noisy data $y_d$ is obtained by adding noise into $y(u_{*})$ pointwise by
$y_{d}(x) = y(u_{*})(x)+\varepsilon\delta\left(y(u_{*})-y(0)\right)(x)$,
where $\delta$ follows the uniform distribution $\mathcal{U}(-1,1)$ and $\varepsilon$ is the relative noise level.
We implement both DFP in \eqref{eqn:DFP} and BFG in \eqref{eqn:BFG}, but show results only for one scheme.
The reconstructions are visualized via heatmaps of the normalized inclusion $\frac{u^k}{\|u^k\|_{L^\infty}}$. The inclusion is indicated by the red color, the background by the blue color, and the boundaries of the exact inclusion $\omega_{j}$ are plotted in black contours.
Unless otherwise specified, Algorithm \ref{alg1} is conducted with the parameter given in Table \ref{table1}.

{
For each example, we present reconstructions at three stages: the initial estimate ($k = 0$), which corresponds to the standard DSM applied to the partial data problem; an intermediate iterate ($k=10$–$15$) where convergence is typically indicated by the U-shape of the damping factor $\lambda_{k,p}$; and the $30$th iterate demonstrating stability over extended iterations.
Our experiments confirm that the IDSM provides a systematic refinement of the index function of DSM, i.e., index function of $k=0$, and the reconstructions achieve a stable configuration from which they do not deteriorate with further iteration.
}

\begin{table}[hbt!]
\centering
\caption{Parameters for numerical experiments}
\label{table1}
\begin{tabular}{ll}
\hline
\textbf{Parameter} & \textbf{Value} \\
\hline
Relative noise level ($\varepsilon$) & $15\%$ \\
Boundary fluxes ($f_1$, $f_2$) & $\sin(4\pi x_1)+0.5$, $\cos(4\pi x_2)+0.5$ \\
Regularization parameter on $\Gamma_D$ ($\alpha_{d}$) & $0.05$ \\
Regularization parameter on $\Gamma_N$ ($\alpha_{n}$) & $2.0$ \\
Integral index ($p$)& $2.0$\\
Mesh $\mathscr{T}_f$ (triangles) & $15728$ \\
Mesh $\mathscr{T}_c$ (triangles) & $1770$ \\
$\gamma$ for conductivity inhomogeneity & $4.0$ \\
$\gamma$ for potential inhomogeneity (linear and nonlinear) & $2.0$ \\
$\gamma$ for mixed inhomogeneity (CE problem) & $3.0$ \\
\hline
\end{tabular}
\end{table}

\subsection{Example 1: EIT}\label{exam1}
The EIT problem is governed by the elliptic system:
\begin{equation*}
\left\{
\begin{aligned}
-\nabla\cdot((1+u_{c})\nabla y)&=0,&\text{in }\Omega,\\
\partial_{n}y&=f,&\text{on }\Gamma,
\end{aligned}
\right.
\end{equation*}
where $u_c$ denotes the conductivity inclusion with $c_i = -0.9$ (i.e., the conductivity is $0.1$ within inclusions).
The physical constraint $ -0.99 \leq u_c \leq 0$ is enforced via the projection $\mathcal{P}(\eta) = \max\{\min\{\eta, 0.0\}, -0.99\}$, with an unknown $c_i$.
Consider two noise levels $\varepsilon_1 = 15\%$ and $\varepsilon_2 = 30\%$.
The reconstructions by the BFG correction scheme are shown in Fig. \ref{fig1.1}.
In Fig. \ref{fig1.2}, we also conduct an ablation study: (i) \textbf{Full-data benchmark}: Replace $\tilde{y}_{d}(u^k)$ with full Cauchy data $y_d$ to evaluate the efficacy of data completion; (ii) \textbf{Homogeneous DtN map}: Replace $\Lambda_{\alpha,D}$ with $\Lambda_\alpha$ (uniform $\alpha$ across $\Gamma$) to show the necessity of heterogeneous regularization; and (iii) \textbf{Unstabilized scheme}: Disable stabilization by setting $\lambda_{k-1,p} \equiv 0$ to examine the impact of the stabilization-correction scheme.

\begin{figure}[hbt!]
\centering
\begin{tabular}{ccccccc}
\includegraphics[width = \sevenfiglen, trim = {0.25cm, 0.25cm, 0.25cm, 0.25cm}, clip]{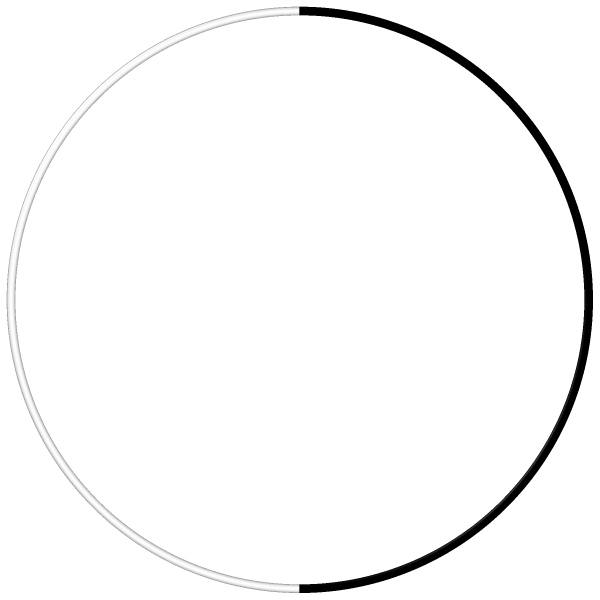}&
\includegraphics[width = \sevenfiglen, trim = {0.25cm, 0.25cm, 0.25cm, 0.25cm}, clip]{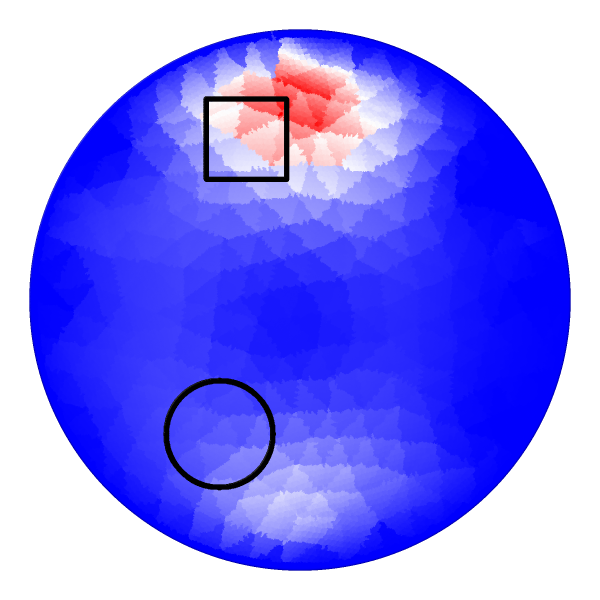}&
\includegraphics[width = \sevenfiglen, trim = {0.25cm, 0.25cm, 0.25cm, 0.25cm}, clip]{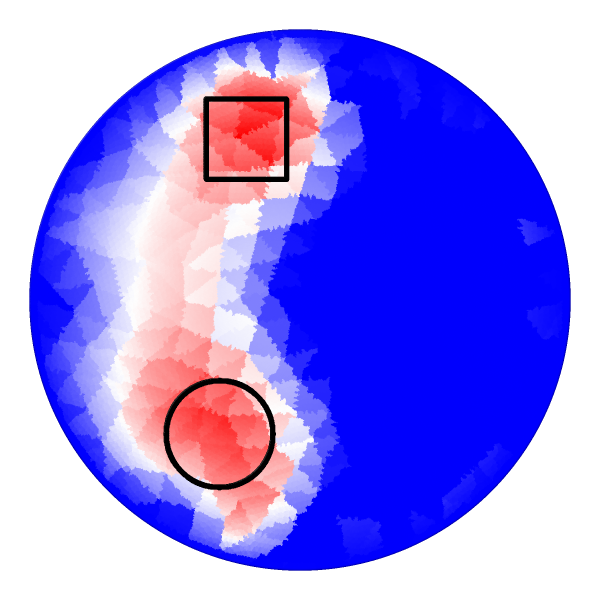}&
\includegraphics[width = \sevenfiglen, trim = {0.25cm, 0.25cm, 0.25cm, 0.25cm}, clip]{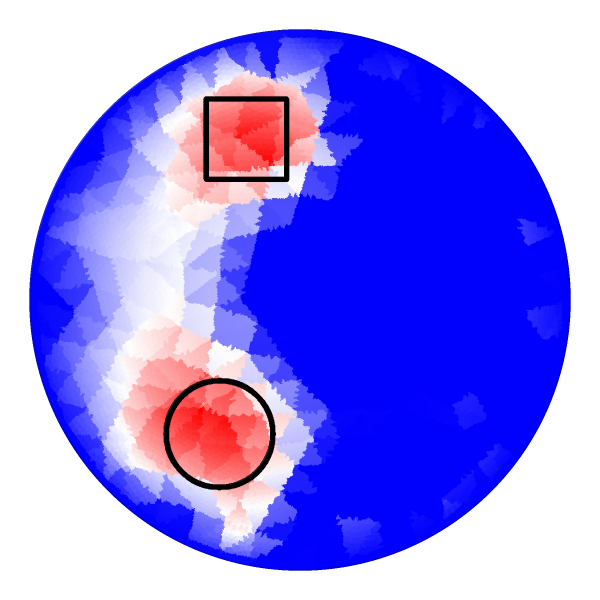}&
\includegraphics[width = \sevenfiglen, trim = {0.25cm, 0.25cm, 0.25cm, 0.25cm}, clip]{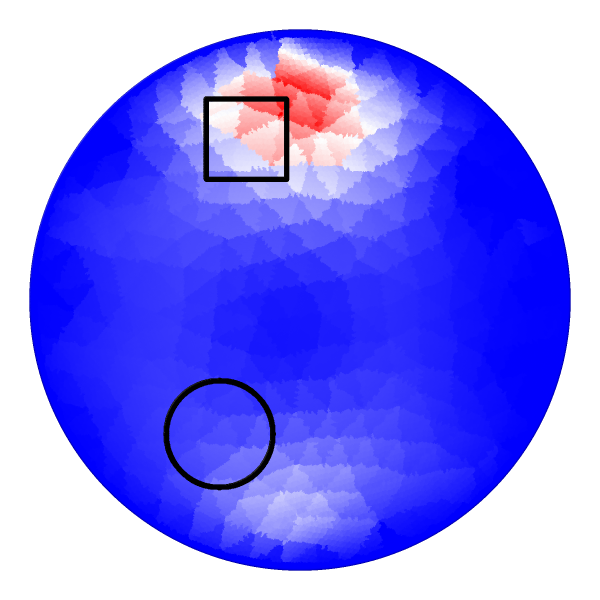}&
\includegraphics[width = \sevenfiglen, trim = {0.25cm, 0.25cm, 0.25cm, 0.25cm}, clip]{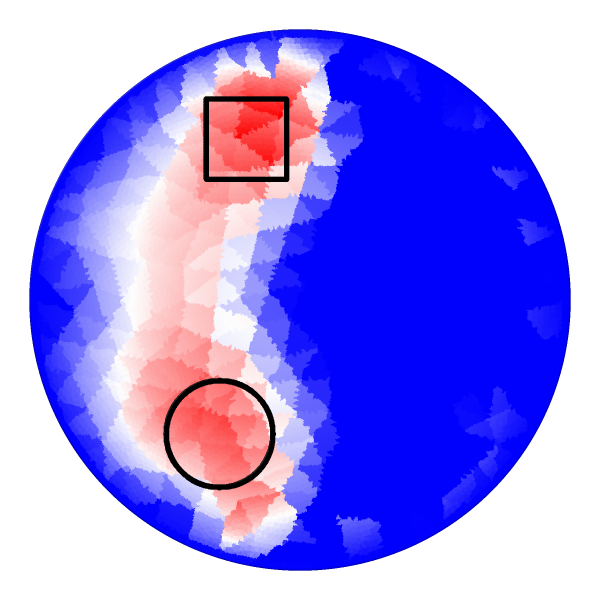}&
\includegraphics[width = \sevenfiglen, trim = {0.25cm, 0.25cm, 0.25cm, 0.25cm}, clip]{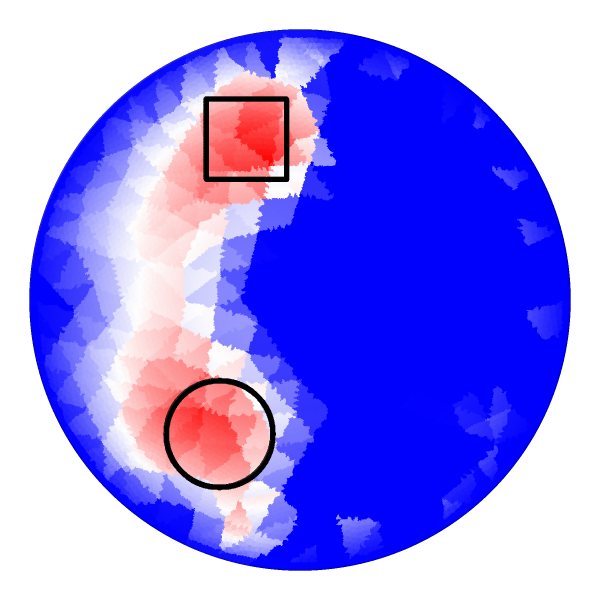}\\
$\Gamma_{D}$& $k=0$ & $k=10$ & $k=30$& $k=0$ & $k=10$ & $k=30$
\end{tabular}
\caption{
The reconstruction of conductivity inclusion $u_c$ for Example \ref{exam1}.
Columns 2–4: Estimates for 15\% noise.
Columns 5–7: Estimates for 30\% noise.
\label{fig1.1}
}
\end{figure}

\begin{figure}[hbt!]
\centering
\begin{tabular}{ccccccc}
\includegraphics[width = \sevenfiglen, trim = {0.25cm, 0.25cm, 0.25cm, 0.25cm}, clip]{figure/split1.png}&
\includegraphics[width = \sevenfiglen, trim = {0.25cm, 0.25cm, 0.25cm, 0.25cm}, clip]{figure/Example4/n3c0000.png}&
\includegraphics[width = \sevenfiglen, trim = {0.25cm, 0.25cm, 0.25cm, 0.25cm}, clip]{figure/Example4/n3c0002.png}&
\includegraphics[width = \sevenfiglen, trim = {0.25cm, 0.25cm, 0.25cm, 0.25cm}, clip]{figure/Example4/n3c0006.png}&
\includegraphics[width = \sevenfiglen, trim = {0.25cm, 0.25cm, 0.25cm, 0.25cm}, clip]{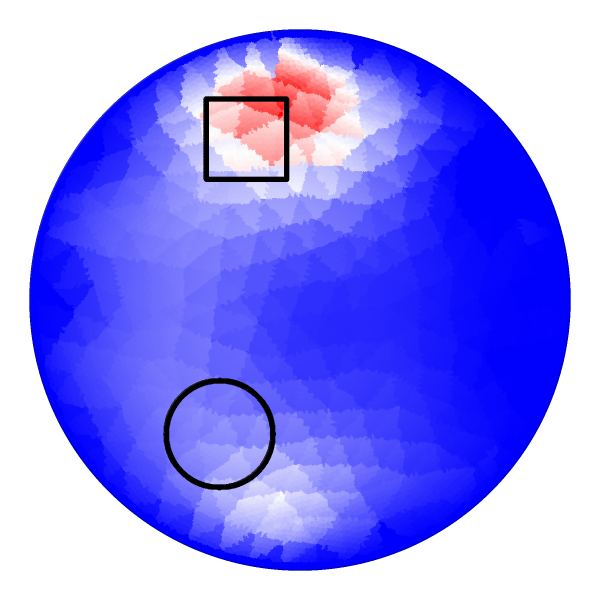}&
\includegraphics[width = \sevenfiglen, trim = {0.25cm, 0.25cm, 0.25cm, 0.25cm}, clip]{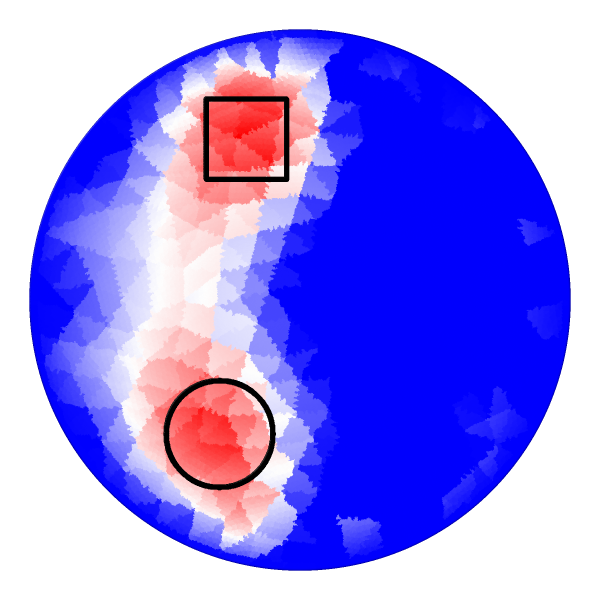}&
\includegraphics[width = \sevenfiglen, trim = {0.25cm, 0.25cm, 0.25cm, 0.25cm}, clip]{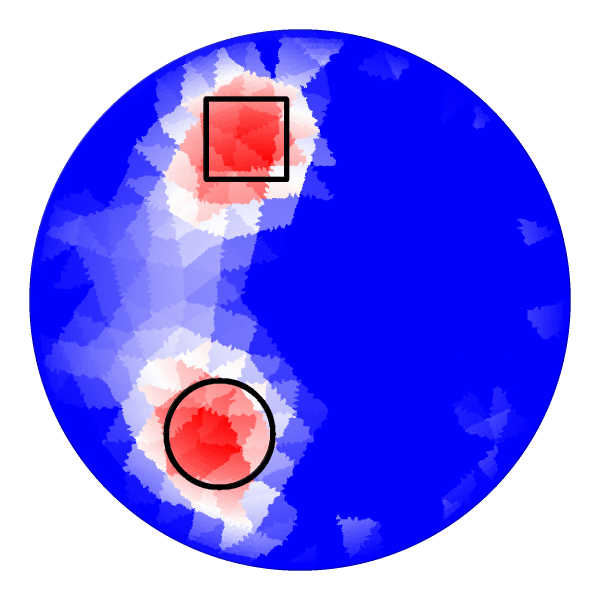}\\
$\Gamma_{D}$& $k=0$ & $k=10$ & $k=30$& $k=0$ & $k=10$ & $k=30$\\
\includegraphics[width = \sevenfiglen, trim = {0.25cm, 0.25cm, 0.25cm, 0.25cm}, clip]{figure/split1.png}&
\includegraphics[width = \sevenfiglen, trim = {0.25cm, 0.25cm, 0.25cm, 0.25cm}, clip]{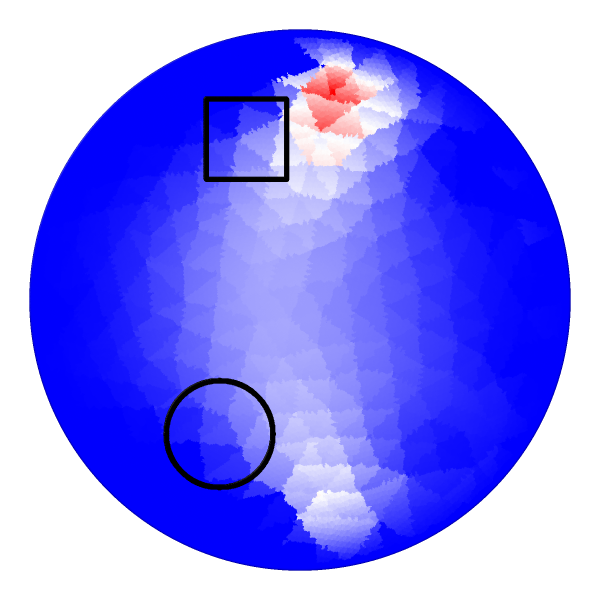}&
\includegraphics[width = \sevenfiglen, trim = {0.25cm, 0.25cm, 0.25cm, 0.25cm}, clip]{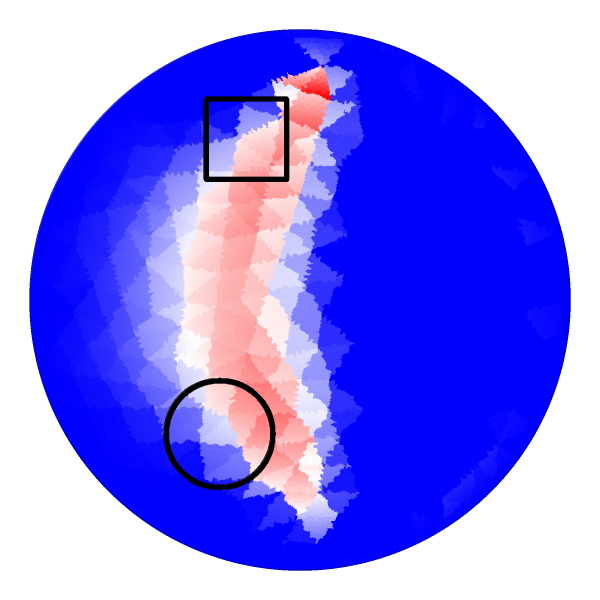}&
\includegraphics[width = \sevenfiglen, trim = {0.25cm, 0.25cm, 0.25cm, 0.25cm}, clip]{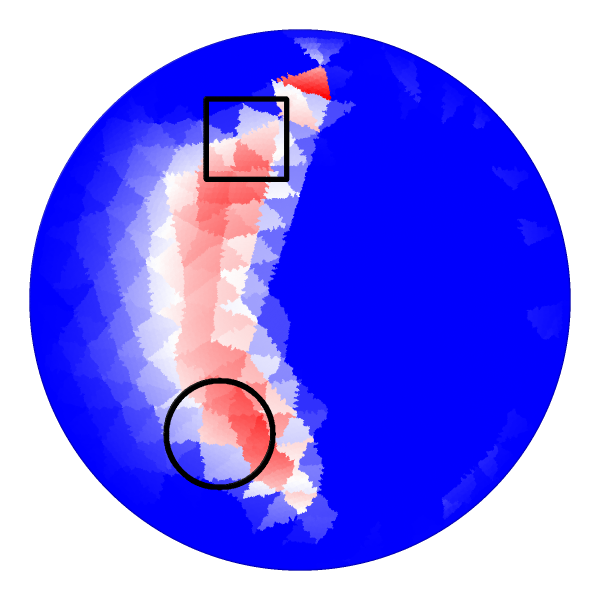}&
\includegraphics[width = \sevenfiglen, trim = {0.25cm, 0.25cm, 0.25cm, 0.25cm}, clip]{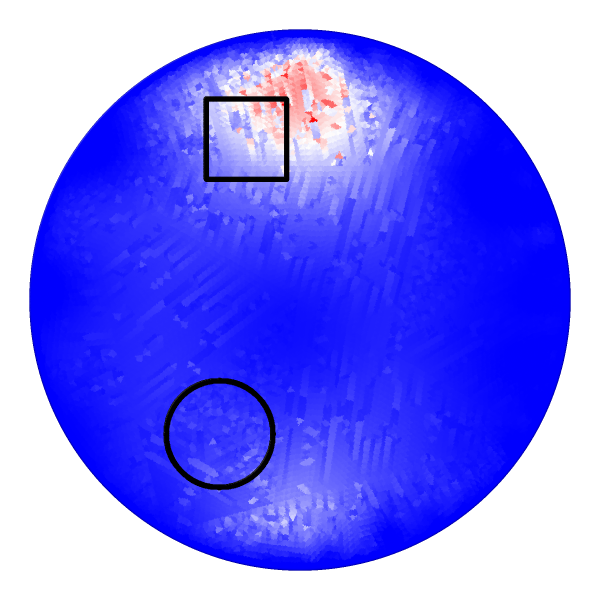}&
\includegraphics[width = \sevenfiglen, trim = {0.25cm, 0.25cm, 0.25cm, 0.25cm}, clip]{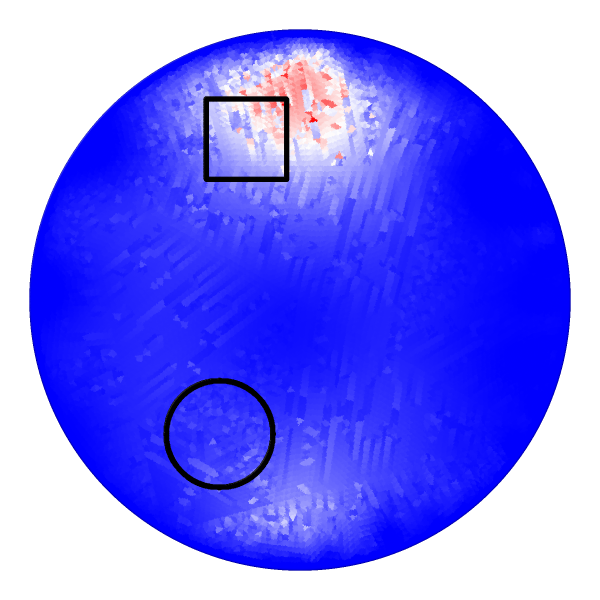}&
\includegraphics[width = \sevenfiglen, trim = {0.25cm, 0.25cm, 0.25cm, 0.25cm}, clip]{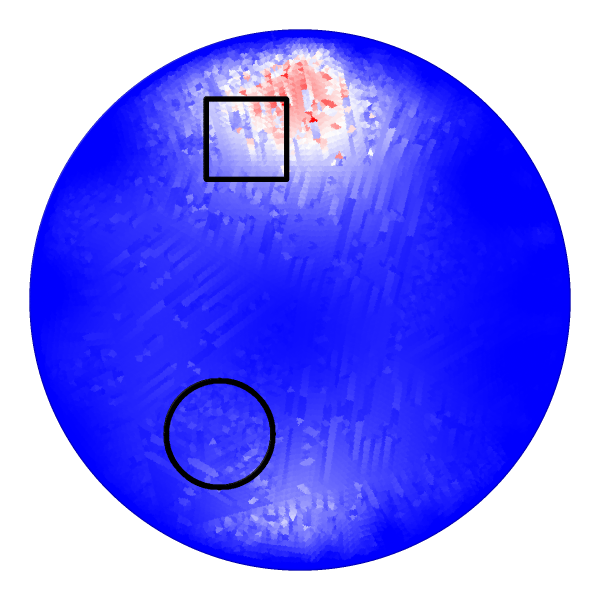}\\
$\Gamma_{D}$& $k=0$ & $k=10$ & $k=30$& $k=0$ & $k=10$ & $k=30$
\end{tabular}
\caption{
The reconstructions for Example \ref{exam1}.
Top row: Partial data (columns 2–4) versus full data (columns 5–7).
Bottom row: Homogeneous DtN map (columns 2–4) and unstabilized scheme (columns 5–7).
\label{fig1.2}
}
\end{figure}

The results show  the robustness of the IDSM.
It accurately localizes the conductivity inclusion despite high noise ($\varepsilon=30\%$) and spatial disjointness (an inclusion resides in the left half of $\Omega$ while measurements are restricted to the right half of $\Gamma$), showing resilience to incomplete data and large noise.
The initial estimate ($k=0$) fails to localize the lower circular inclusion, while the iterative refinement achieves reliable localization by $k=10$.
Under $\varepsilon=30\%$, the estimate remains stable up to $k=30$ iterations, confirming the stability of the iterative correction.

The ablation study confirms the necessity of each innovation:
The close match between the results for partial data and the full-data benchmark (Fig. \ref{fig1.2}, top row, columns 5–7) validates the data completion scheme.
{While the full-data reconstruction, benefiting from the complete scattering field, yields a more distinct image of the two inclusions, the partial data estimate reconstructed without information on the inaccessible boundary is remarkably comparable, showing its capability to effectively reconstruct missing boundary information.}
The superior accuracy of the HR-DtN map over the homogeneous DtN map (Fig. \ref{fig1.2}, bottom row, columns 2–4) underscores the crucial role of heterogeneous regularization, which enhances the sensitivity on $\Gamma_D$ while suppressing noise on $\Gamma_N$.
The pronounced inaccuracies and slow convergence of the unstabilized scheme (Fig. \ref{fig1.2}, bottom row, columns 5–7) highlight the indispensable role of the stabilization-correction scheme in controlling error propagation and ensuring stability {over extended iterations}.
These results show that data completion, HR-DtN map, and stabilization-correction work synergistically to achieve robust reconstructions under partial data constraints.

\subsection{Example 2: DOT model}\label{exam2}

This example examines the IDSM on DOT in \eqref{eqn3}.
Within their inclusions, the conductivity $c = 0.1$, while the potential $a = 10$; both have unit background ($c_0 = 1$, $p_0 = 1$).
The pointwise projection $\mathcal{P}$ is set to $
\mathcal{P}(\eta_c)=\max\{\min\{\eta_{c}, 0.0\}, -0.99\}$ and $\mathcal{P}(\eta_{p})=
\min\{\max\{\eta_{p}, 0.0\}, 19\}$.
Consider recovering overlapping inclusions of distinct types under three configurations of $\Gamma_D$: two quarter-circles, one third-circle, and one quarter-circle.
The reconstructions using DFP correction are shown in Fig. \ref{fig2}.

\begin{figure}[hbt!]
\centering
\begin{tabular}{ccccccc}
\includegraphics[width = \sevenfiglen, trim = {0.25cm, 0.25cm, 0.25cm, 0.25cm}, clip]{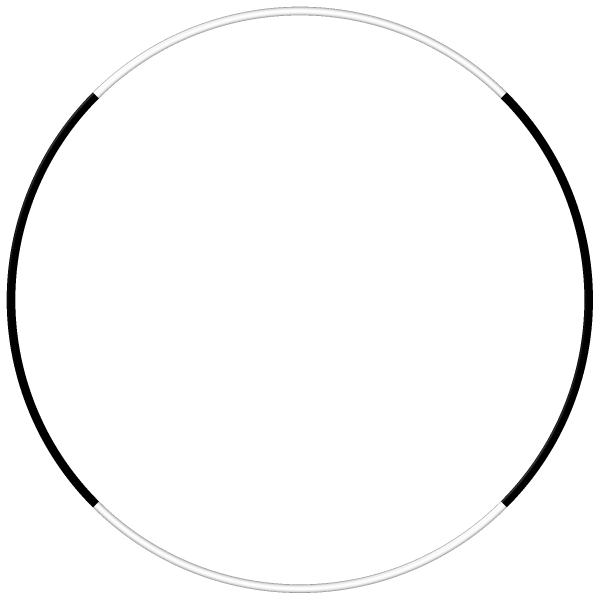}&
\includegraphics[width = \sevenfiglen, trim = {0.25cm, 0.25cm, 0.25cm, 0.25cm}, clip]{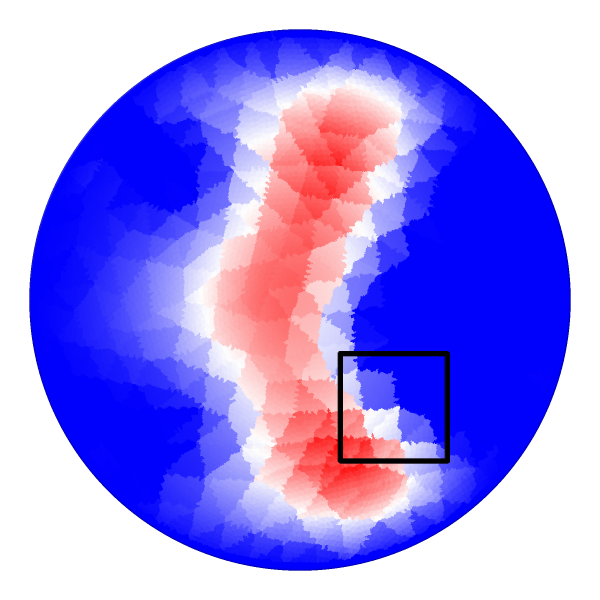}&
\includegraphics[width = \sevenfiglen, trim = {0.25cm, 0.25cm, 0.25cm, 0.25cm}, clip]{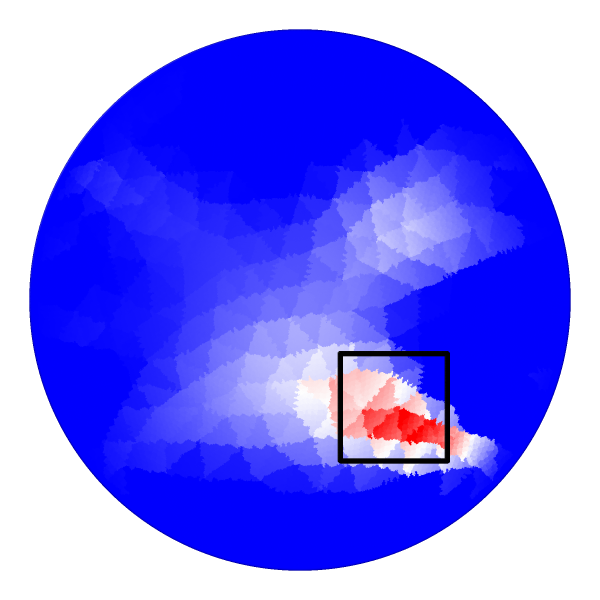}&
\includegraphics[width = \sevenfiglen, trim = {0.25cm, 0.25cm, 0.25cm, 0.25cm}, clip]{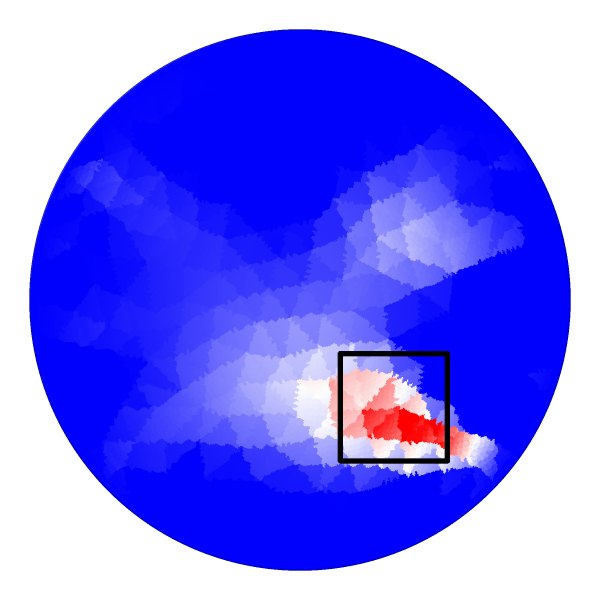}&
\includegraphics[width = \sevenfiglen, trim = {0.25cm, 0.25cm, 0.25cm, 0.25cm}, clip]{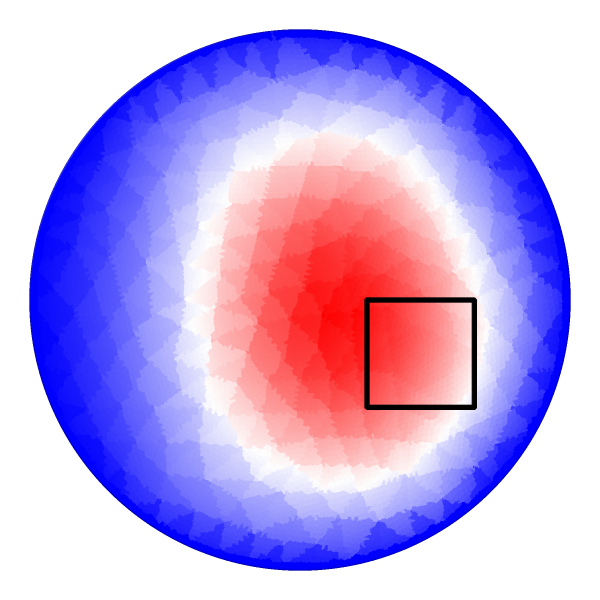}&
\includegraphics[width = \sevenfiglen, trim = {0.25cm, 0.25cm, 0.25cm, 0.25cm}, clip]{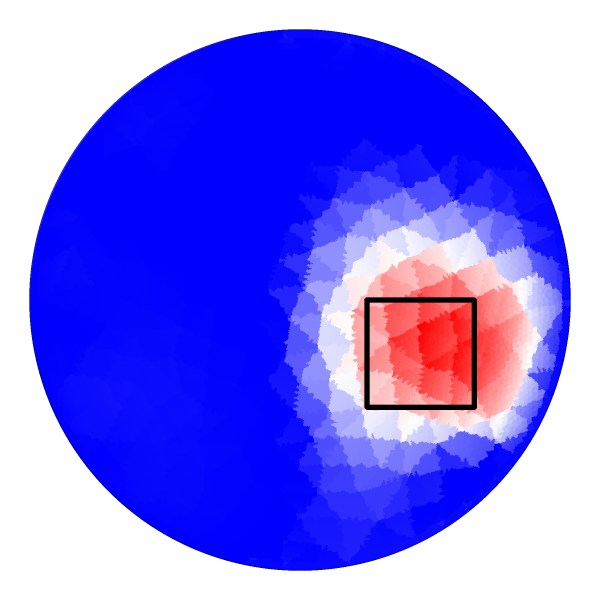}&
\includegraphics[width = \sevenfiglen, trim = {0.25cm, 0.25cm, 0.25cm, 0.25cm}, clip]{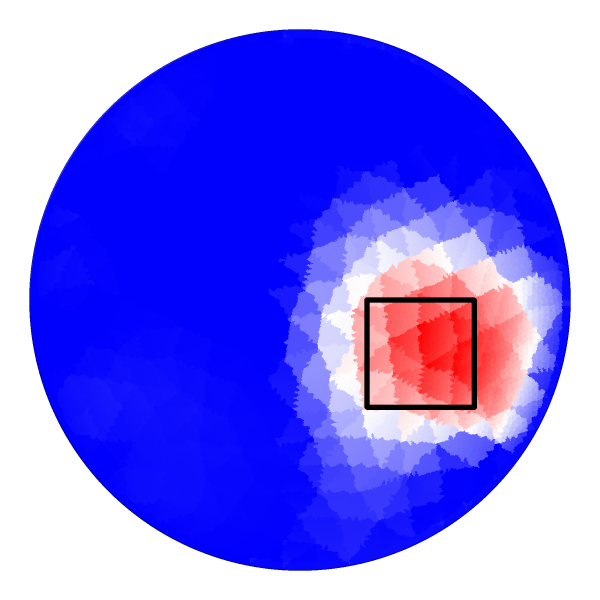}\\
\includegraphics[width = \sevenfiglen, trim = {0.25cm, 0.25cm, 0.25cm, 0.25cm}, clip]{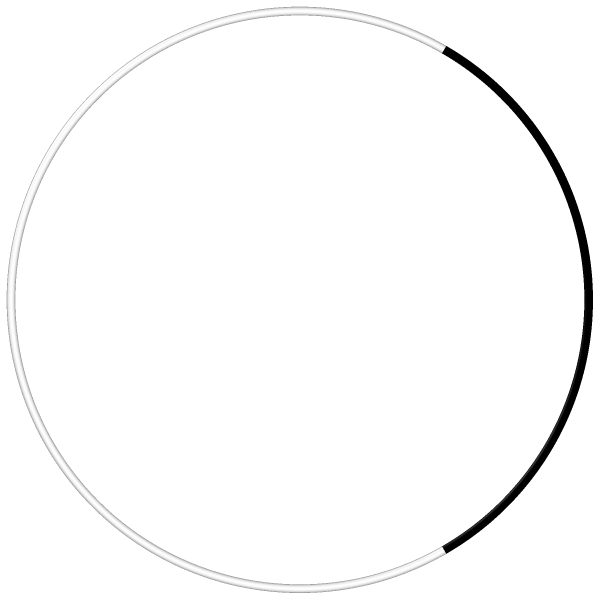}&
\includegraphics[width = \sevenfiglen, trim = {0.25cm, 0.25cm, 0.25cm, 0.25cm}, clip]{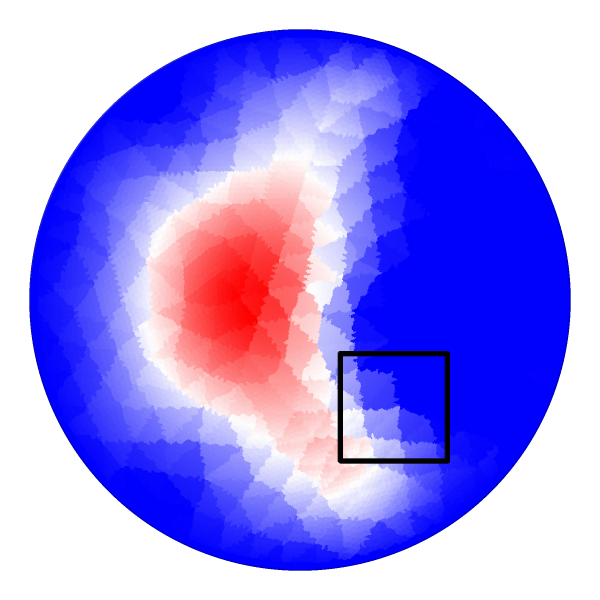}&
\includegraphics[width = \sevenfiglen, trim = {0.25cm, 0.25cm, 0.25cm, 0.25cm}, clip]{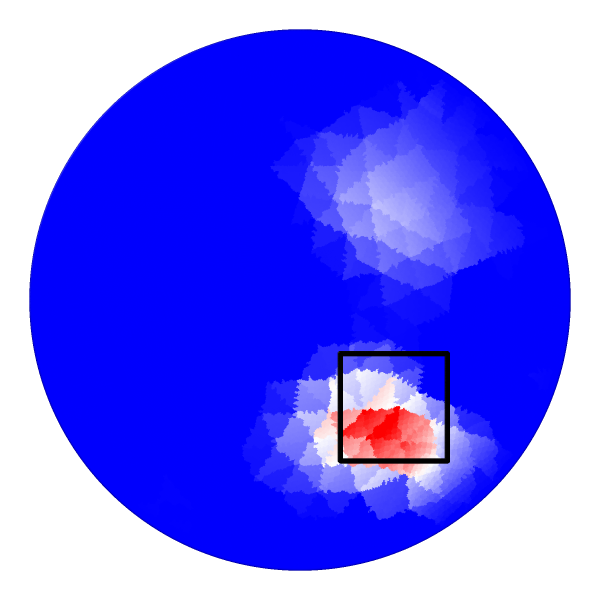}&
\includegraphics[width = \sevenfiglen, trim = {0.25cm, 0.25cm, 0.25cm, 0.25cm}, clip]{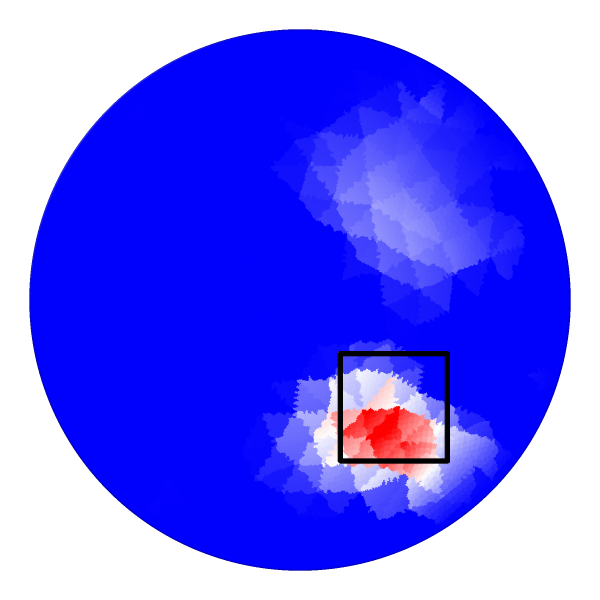}&
\includegraphics[width = \sevenfiglen, trim = {0.25cm, 0.25cm, 0.25cm, 0.25cm}, clip]{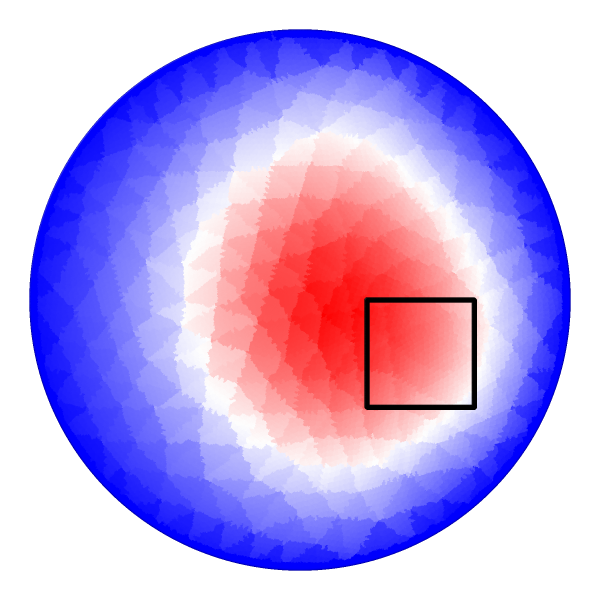}&
\includegraphics[width = \sevenfiglen, trim = {0.25cm, 0.25cm, 0.25cm, 0.25cm}, clip]{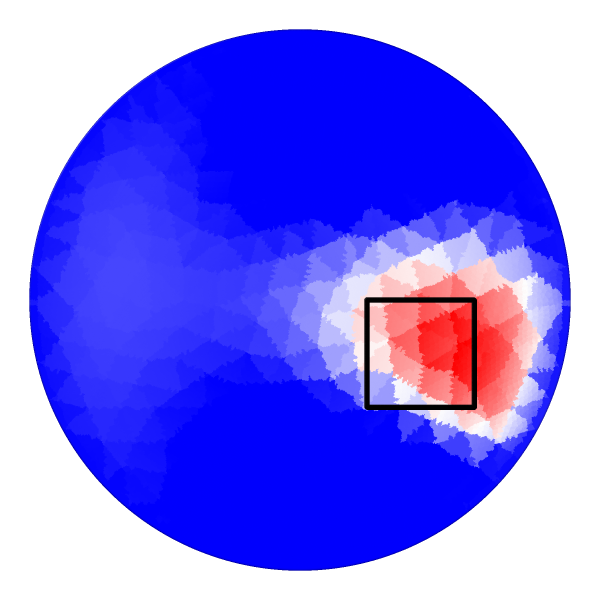}&
\includegraphics[width = \sevenfiglen, trim = {0.25cm, 0.25cm, 0.25cm, 0.25cm}, clip]{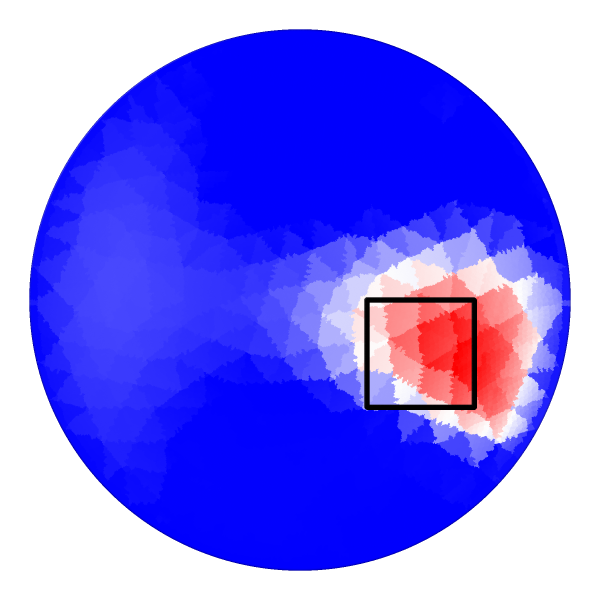}\\
\includegraphics[width = \sevenfiglen, trim = {0.25cm, 0.25cm, 0.25cm, 0.25cm}, clip]{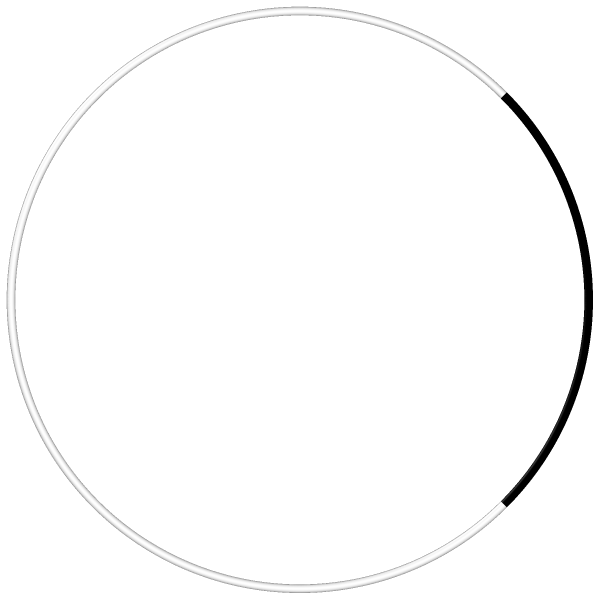}&
\includegraphics[width = \sevenfiglen, trim = {0.25cm, 0.25cm, 0.25cm, 0.25cm}, clip]{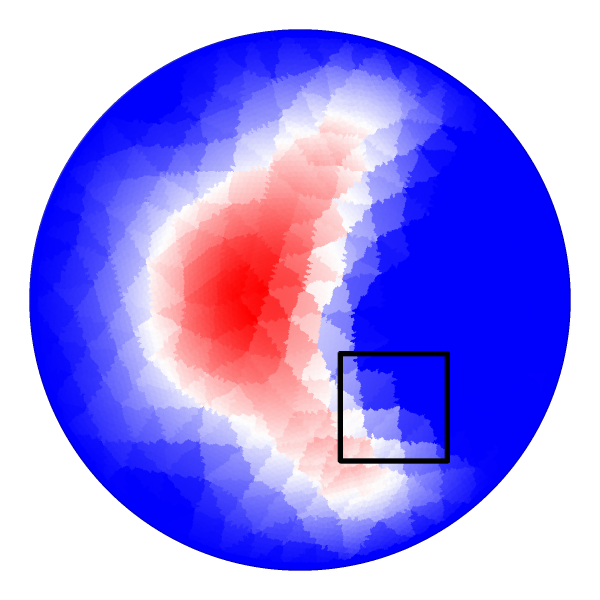}&
\includegraphics[width = \sevenfiglen, trim = {0.25cm, 0.25cm, 0.25cm, 0.25cm}, clip]{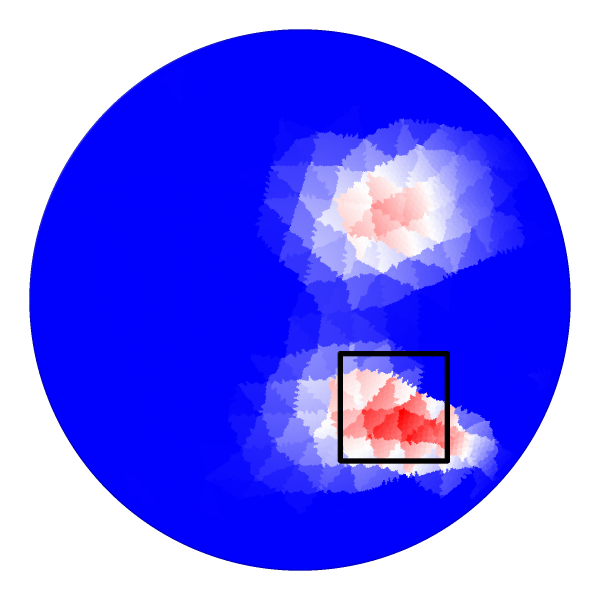}&
\includegraphics[width = \sevenfiglen, trim = {0.25cm, 0.25cm, 0.25cm, 0.25cm}, clip]{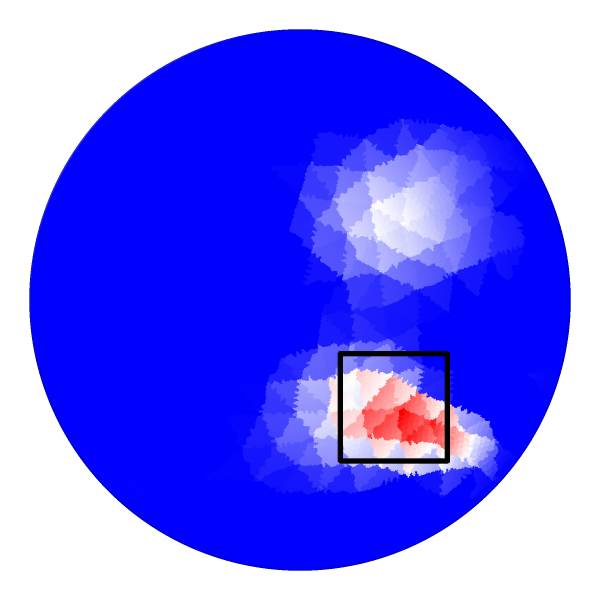}&
\includegraphics[width = \sevenfiglen, trim = {0.25cm, 0.25cm, 0.25cm, 0.25cm}, clip]{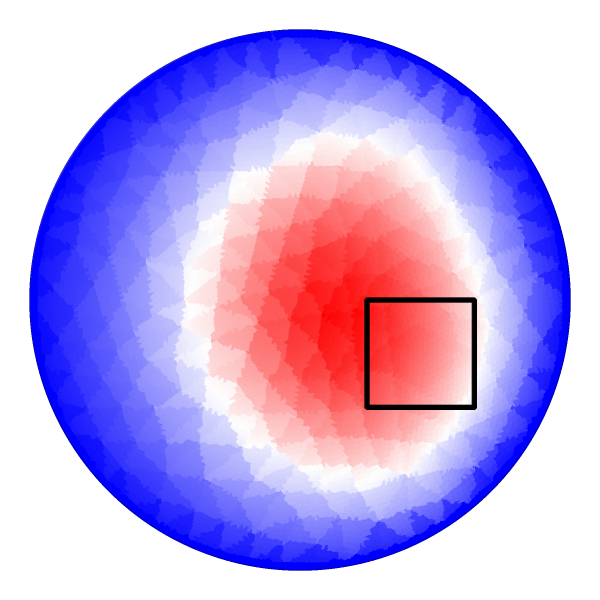}&
\includegraphics[width = \sevenfiglen, trim = {0.25cm, 0.25cm, 0.25cm, 0.25cm}, clip]{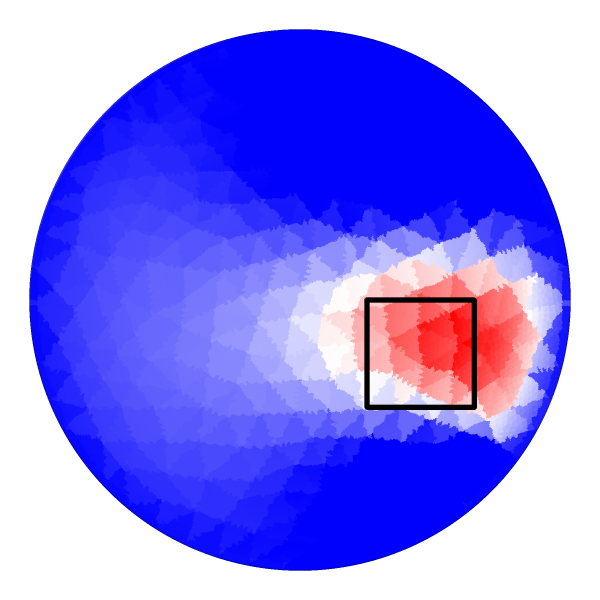}&
\includegraphics[width = \sevenfiglen, trim = {0.25cm, 0.25cm, 0.25cm, 0.25cm}, clip]{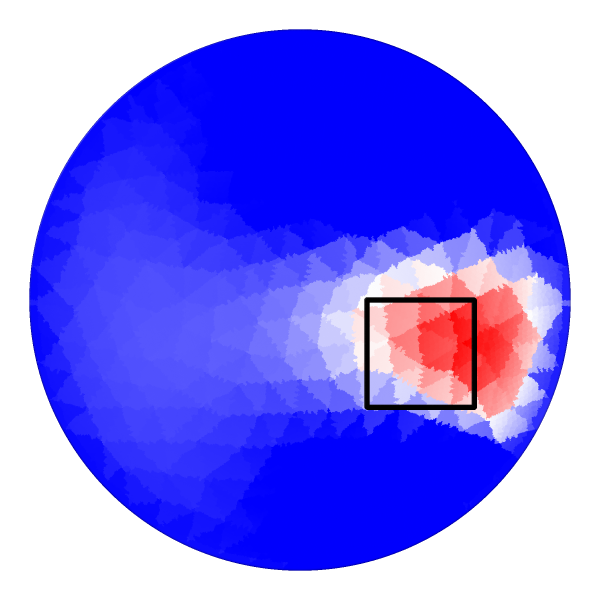}\\
$\Gamma_{D}$& $k=0$ & $k=20$ & $k=30$& $k=0$ & $k=20$ & $k=30$
\end{tabular}
\caption{
The reconstructed conductivity ($u_c$, columns 2–4) and potential ($u_{p}$, columns 5–7) inclusions for Example \ref{exam2}.
\label{fig2}
}
\end{figure}

The IDSM correctly localizes the conductivity and potential inclusions despite their overlapping geometries.
The initial estimate ($k=0$) of the conductivity inclusion falsely identifies an artifact in the central of the domain, which is corrected by $k=20$.
The initial potential estimate $u_p^0$ exhibits severe inaccuracies across the domain $\Omega$, yet the iterative refinement suppresses errors and correctly identifies the inclusion by $k=20$.
The method performs robustly for all configurations of $\Gamma_D$.
{As expected, the reconstruction quality improves with the length of the measurable arc $\Gamma_D$ (the second row verse the third row); nevertheless, all configurations yield satisfactory estimates by $k=30$, demonstrating the robustness of the method to varying data availability.}

\subsection{Example 3: DOT model}\label{exam3}
{This example revisits the DOT problem (cf. \eqref{eqn3}) to evaluate the robustness of the IDSM with respect to the integrability index $p$, which governs the damping of low-rank corrections in the stabilization scheme \eqref{eqn33}.
We test two values: $p_1 = 1$ and $p_2 = 99$. The reconstructions by the BFG correction scheme are shown in Fig. \ref{fig3}.}

\begin{figure}[hbt!]
\centering
\begin{tabular}{ccccccc}
\includegraphics[width = \sevenfiglen, trim = {0.25cm, 0.25cm, 0.25cm, 0.25cm}, clip]{figure/split1.png}&
\includegraphics[width = \sevenfiglen, trim = {0.25cm, 0.25cm, 0.25cm, 0.25cm}, clip]{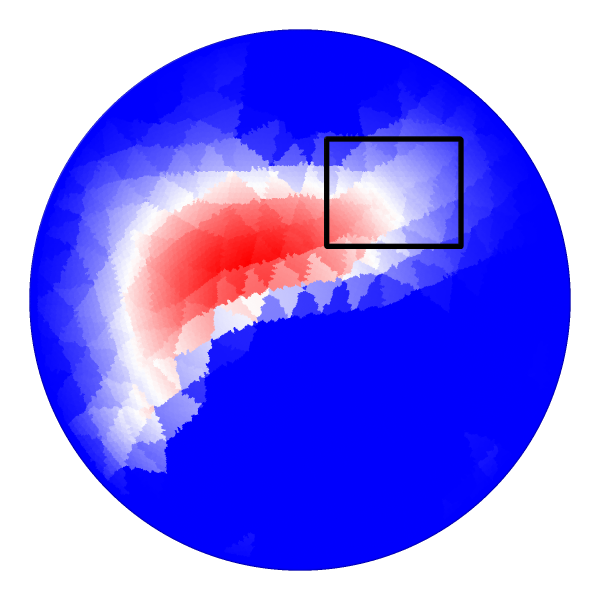}&
\includegraphics[width = \sevenfiglen, trim = {0.25cm, 0.25cm, 0.25cm, 0.25cm}, clip]{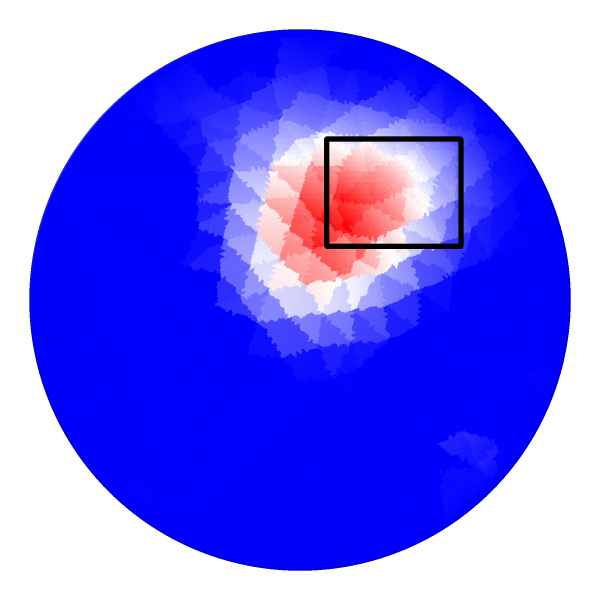}&
\includegraphics[width = \sevenfiglen, trim = {0.25cm, 0.25cm, 0.25cm, 0.25cm}, clip]{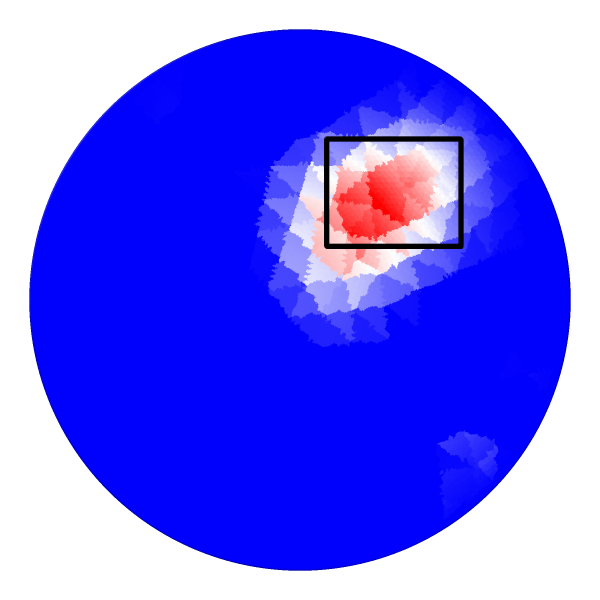}&
\includegraphics[width = \sevenfiglen, trim = {0.25cm, 0.25cm, 0.25cm, 0.25cm}, clip]{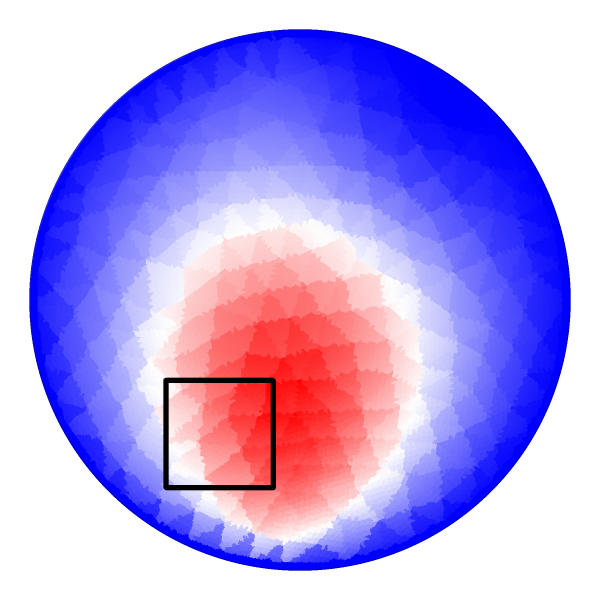}&
\includegraphics[width = \sevenfiglen, trim = {0.25cm, 0.25cm, 0.25cm, 0.25cm}, clip]{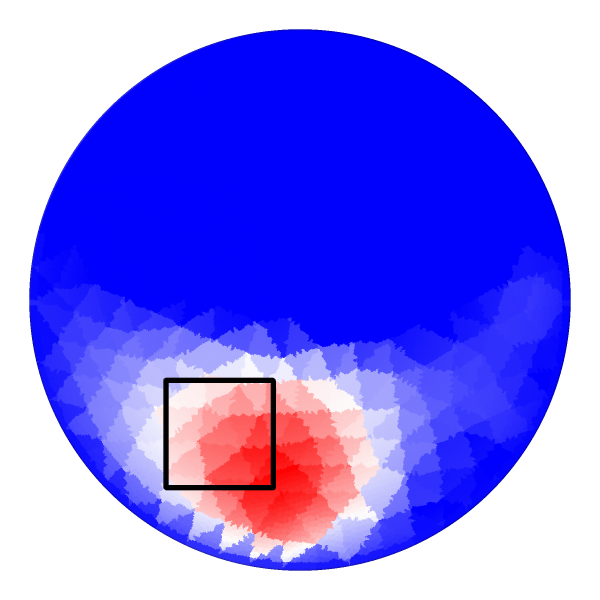}&
\includegraphics[width = \sevenfiglen, trim = {0.25cm, 0.25cm, 0.25cm, 0.25cm}, clip]{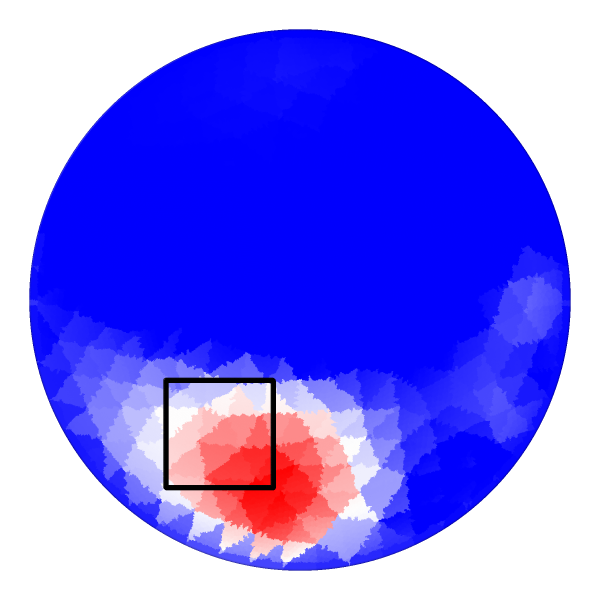}\\
\includegraphics[width = \sevenfiglen, trim = {0.25cm, 0.25cm, 0.25cm, 0.25cm}, clip]{figure/split1.png}&
\includegraphics[width = \sevenfiglen, trim = {0.25cm, 0.25cm, 0.25cm, 0.25cm}, clip]{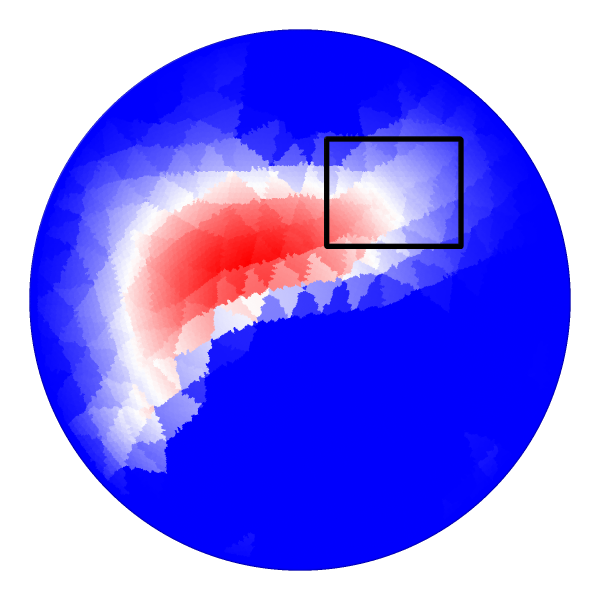}&
\includegraphics[width = \sevenfiglen, trim = {0.25cm, 0.25cm, 0.25cm, 0.25cm}, clip]{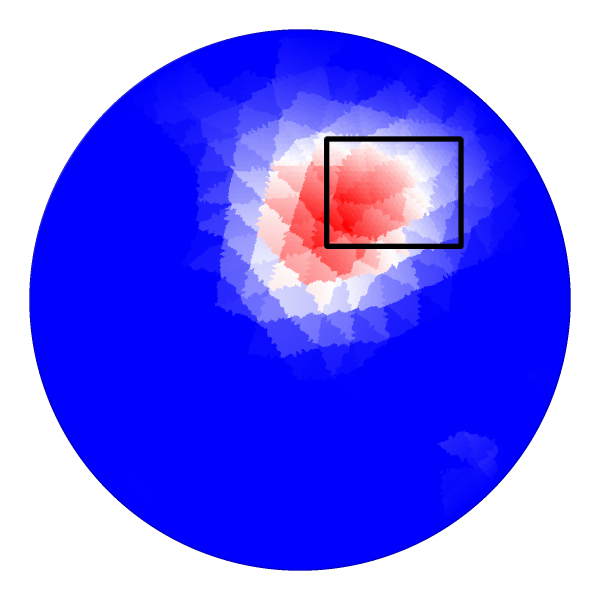}&
\includegraphics[width = \sevenfiglen, trim = {0.25cm, 0.25cm, 0.25cm, 0.25cm}, clip]{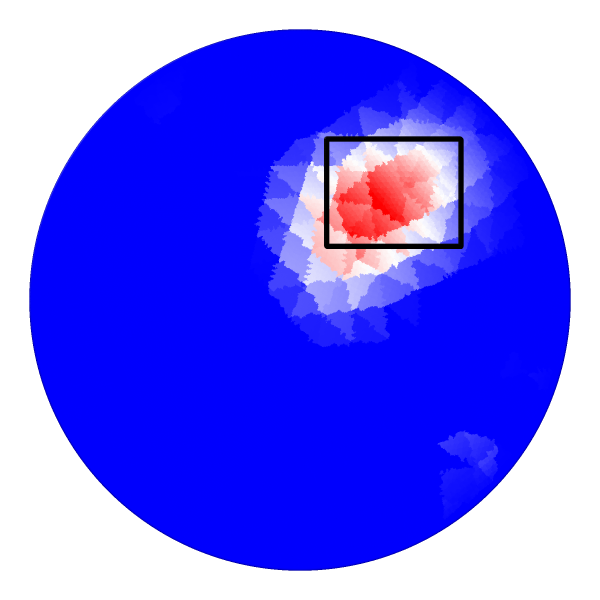}&
\includegraphics[width = \sevenfiglen, trim = {0.25cm, 0.25cm, 0.25cm, 0.25cm}, clip]{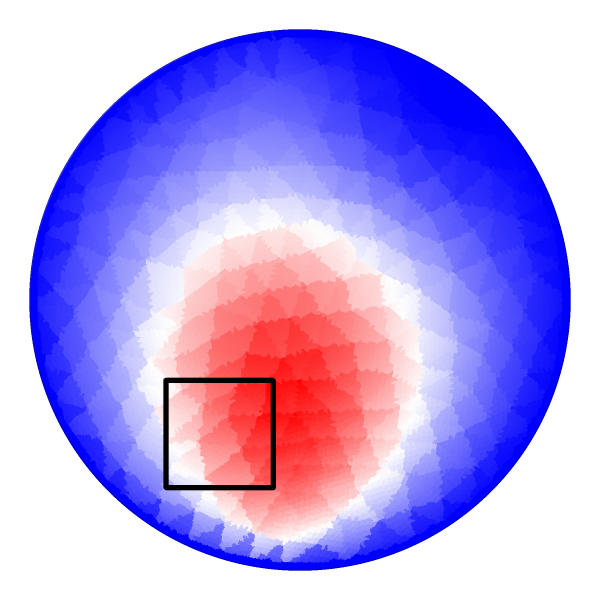}&
\includegraphics[width = \sevenfiglen, trim = {0.25cm, 0.25cm, 0.25cm, 0.25cm}, clip]{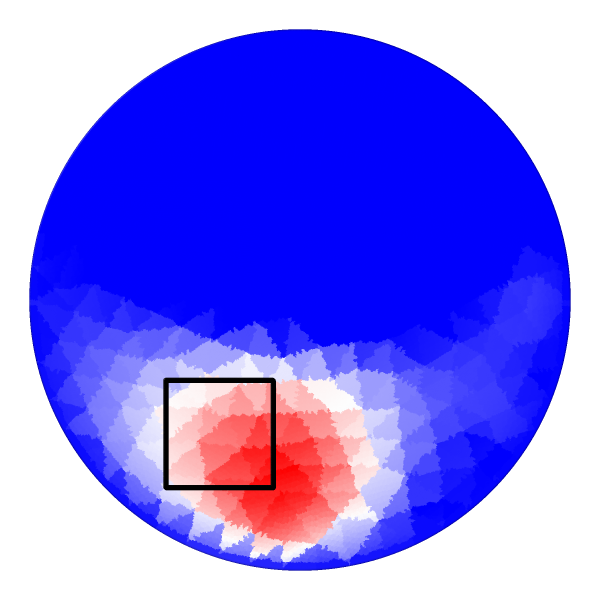}&
\includegraphics[width = \sevenfiglen, trim = {0.25cm, 0.25cm, 0.25cm, 0.25cm}, clip]{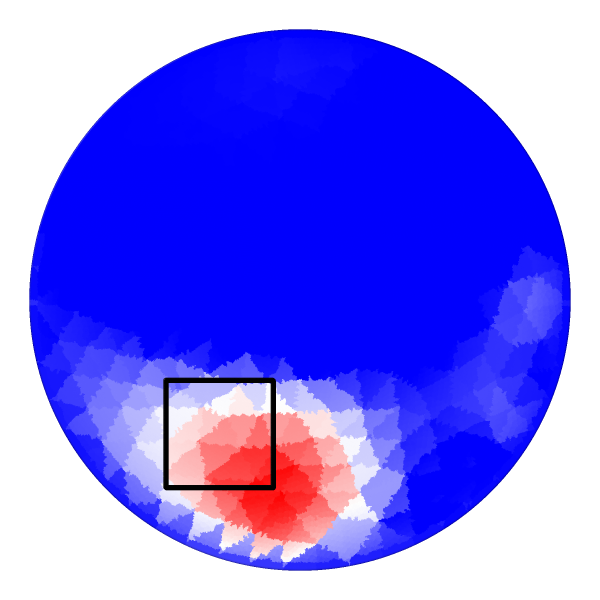}\\
$\Gamma_{D}$& $k=0$ & $k=15$ & $k=30$& $k=0$ & $k=15$ & $k=30$
\end{tabular}
\caption{
The reconstructions for Example \ref{exam3} with integral indices $p=1$ (top) and $p=99$ (bottom).
Columns 2–4: Conductivity inclusion $u_c$.
Columns 5–7: Potential inclusion $u_{p}$.
\label{fig3}
}
\end{figure}

The IDSM accurately identifies both conductivity and potential inclusions for both $p$ values.
The conductivity inclusion (right half, near $\Gamma_{D}$) exhibits sharp localization.
The potential inclusion (left half, far to $\Gamma_{D}$) is clearly delineated without cross-talk artifacts.
The reconstructions remain consistent between $p_{1}$ and $p_{2}$, showing negligible sensitivity of the convergence speed, accuracy, or artifact suppression with respect to the exponent $p$.

\subsection{Example 4: CE model}\label{exam4}
This example evaluates the IDSM on a cardiac electrophysiology (CE) model governed by the semilinear elliptic system \cite{Scacchi2014}:
\begin{equation*}
\left\{
\begin{aligned}
-\nabla \cdot \left( \widetilde{\sigma} \nabla y \right) + \chi_{\Omega \setminus \overline{\omega}} y^3 &= f, & \text{in } \Omega, \\
\widetilde{\sigma} \partial_{n}y &= 0, & \text{on } \Gamma,
\end{aligned}
\right.
\end{equation*}
with $\widetilde{\sigma}(x) = 10^{-4}$ in the ischemic region $\overline{\omega}$ and $\widetilde{\sigma}(x) = 1$ elsewhere.
The goal is to recover the characteristic function $u_* = \chi_{\omega}$ from boundary measurements.
We apply two boundary fluxes $f_1 = 1.1 - x_2^2$ and $f_2 = x_2^2$ to satisfy the positivity condition \cite[Assumption 1.2]{jin2025adaptive}.
We enforce the constraint $0 \leq u^k \leq 1$ via the projection $\mathcal{P}(\eta) = \max\{\min\{\eta, 1.0\}, 0.0\}$.
We test two HR-DtN parameter sets: $(\alpha_{d,1}, \alpha_{n,1}) = (1.0\times 10^{-3}, 2.0)$ and $(\alpha_{d,2}, \alpha_{n,2}) = (0.05, 10.0)$.
The reconstructions by the DFP correction are shown in Fig. \ref{fig4}.

\begin{figure}[hbt!]
\centering
\begin{tabular}{ccccccc}
\includegraphics[width = \sevenfiglen, trim = {0.25cm, 0.25cm, 0.25cm, 0.25cm}, clip]{figure/split1.png}&
\includegraphics[width = \sevenfiglen, trim = {0.25cm, 0.25cm, 0.25cm, 0.25cm}, clip]{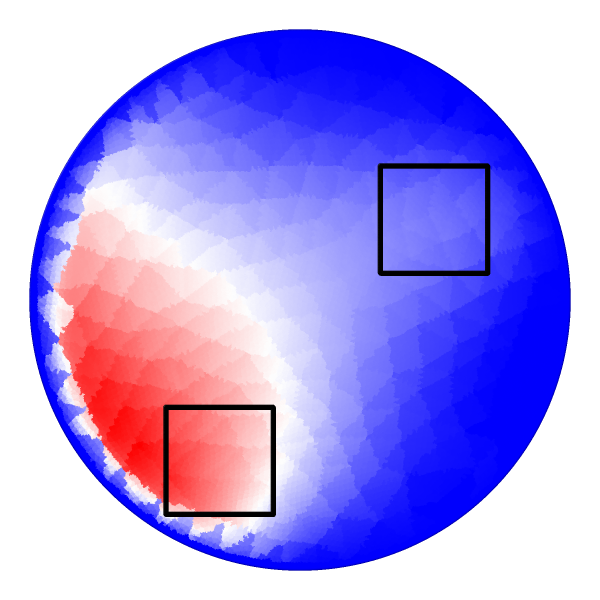}&
\includegraphics[width = \sevenfiglen, trim = {0.25cm, 0.25cm, 0.25cm, 0.25cm}, clip]{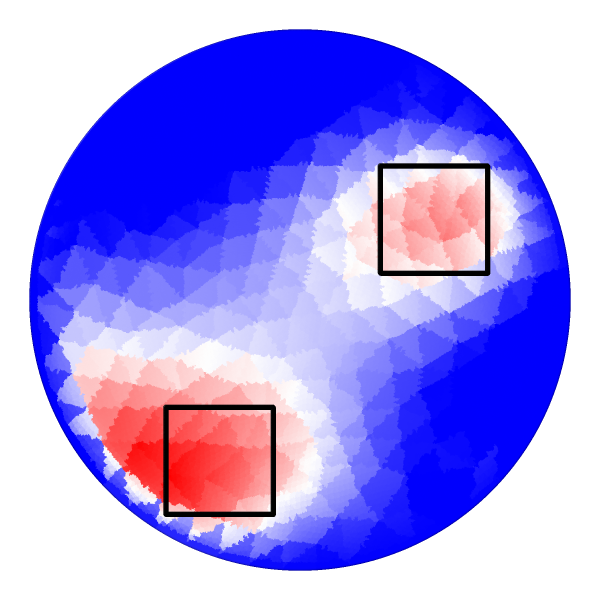}&
\includegraphics[width = \sevenfiglen, trim = {0.25cm, 0.25cm, 0.25cm, 0.25cm}, clip]{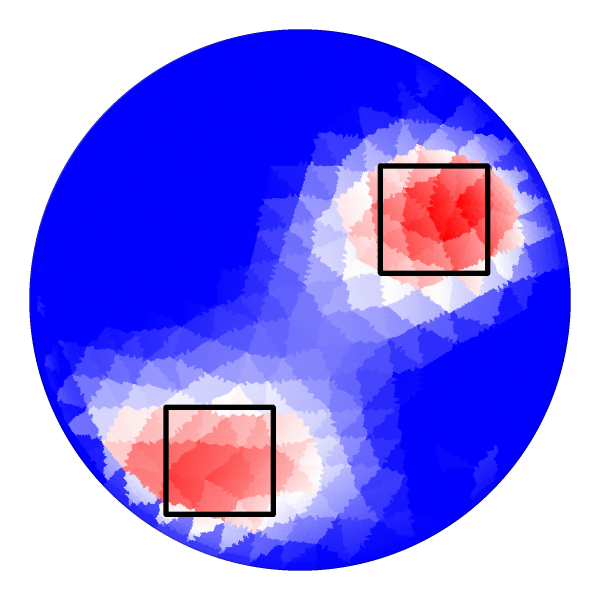}&
\includegraphics[width = \sevenfiglen, trim = {0.25cm, 0.25cm, 0.25cm, 0.25cm}, clip]{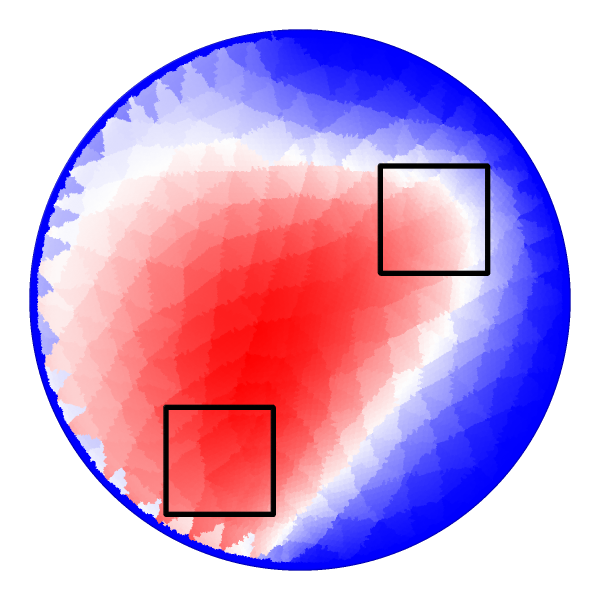}&
\includegraphics[width = \sevenfiglen, trim = {0.25cm, 0.25cm, 0.25cm, 0.25cm}, clip]{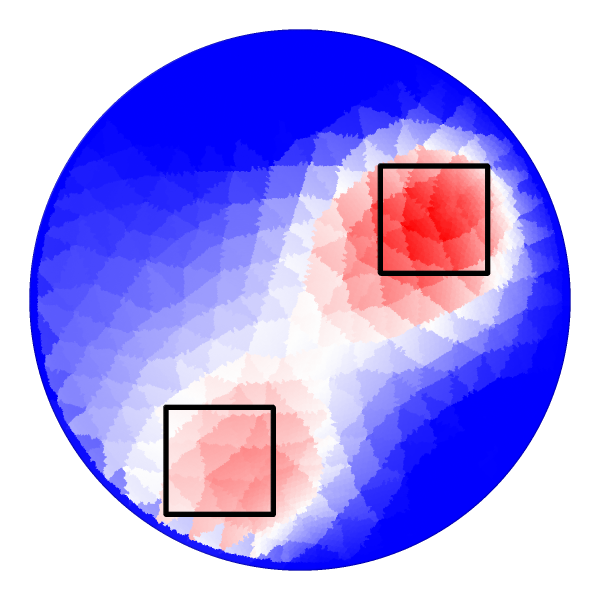}&
\includegraphics[width = \sevenfiglen, trim = {0.25cm, 0.25cm, 0.25cm, 0.25cm}, clip]{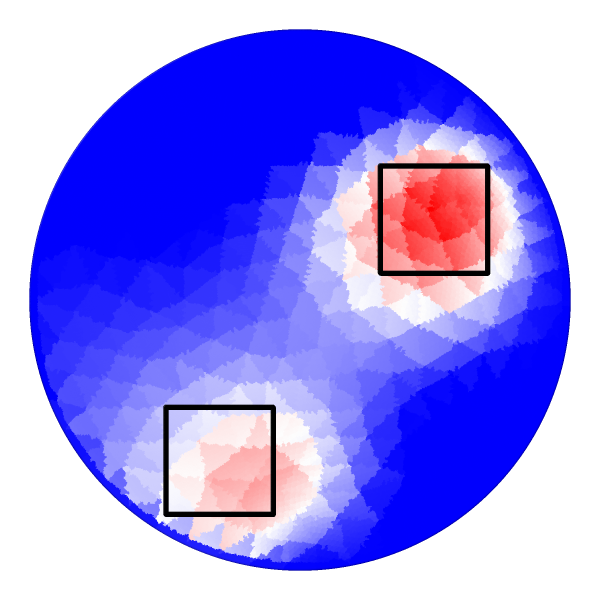}\\
$\Gamma_{D}$& $k=0$ & $k=15$ & $k=30$& $k=0$ & $k=15$ & $k=30$
\end{tabular}
\caption{
The reconstruction of characteristic function $u$ for Example \ref{exam4}.
Columns 2–4: Estimates for $\alpha_{d}=10^{-3}$ and $\alpha_{n}=2.0$.
Columns 5–7: Estimates for $\alpha_{d}=0.05$ and $\alpha_{n}=10.0$.
\label{fig4}
}
\end{figure}

{The iterative refinement of the IDSM effectively overcomes the limitations of the initial index functions.
For the parameter set $(\alpha_{d,1}, \alpha_{n,1})$, the initial estimate ($k=0$) fails to detect the inclusion in the top-right of the domain, while for $(\alpha_{d,2}, \alpha_{n,2})$, it incorrectly merges the two disjoint inclusions into a single, large artifact.
However, in both cases, the iterative process successfully corrects these inaccuracies, suppressing the spurious artifact and recovering the correct, separate geometries by $k=15$.
This demonstrates the robustness of the method, as the well-converged final reconstructions exhibit negligible sensitivity to the HR-DtN parameters $(\alpha_d, \alpha_n)$ in accuracy and stability, despite significant dependence on these parameters in the initial estimates.}

\subsection{Example 5: {Modulus nonlinearity} model}\label{exam5}

This example assesses the IDSM for a nonlinear elliptic inverse problem governed by
\begin{equation*}
\left\{
\begin{aligned}
-\Delta y + y + u|y|y &= 0, && \text{in } \Omega, \\
\partial_{n}y &= f, && \text{on } \Gamma,
\end{aligned}
\right.
\end{equation*}
where the potential $u$ is $40$ within inclusions and $0$ in the background.
The constraint $0 \leq u \leq 60$ is enforced via the projection $\mathcal{P}(\eta) = \max\{\min\{\eta, 60\}, 0\}$.
The reconstructions uses one pair of partial Cauchy data $(f_1, y_d|_{\Gamma_D})$, with boundary flux $f_1(x) = x_1^2$ and accessible set $\Gamma_D$ comprising two disconnected boundary arcs.
The reconstructions obtained with the BFG correction scheme are presented in Fig. \ref{fig5}.

\begin{figure}[hbt!]
\centering
\begin{tabular}{cccc}
\includegraphics[width = \sevenfiglen, trim = {0.25cm, 0.25cm, 0.25cm, 0.25cm}, clip]{figure/split2.png}&
\includegraphics[width = \sevenfiglen, trim = {0.25cm, 0.25cm, 0.25cm, 0.25cm}, clip]{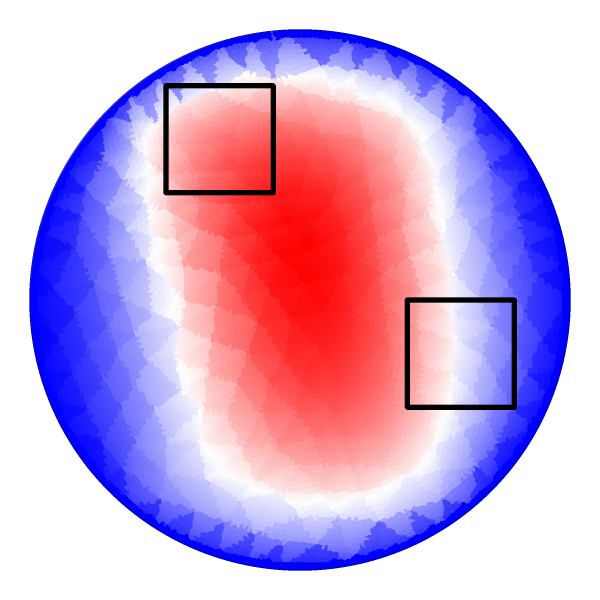}&
\includegraphics[width = \sevenfiglen, trim = {0.25cm, 0.25cm, 0.25cm, 0.25cm}, clip]{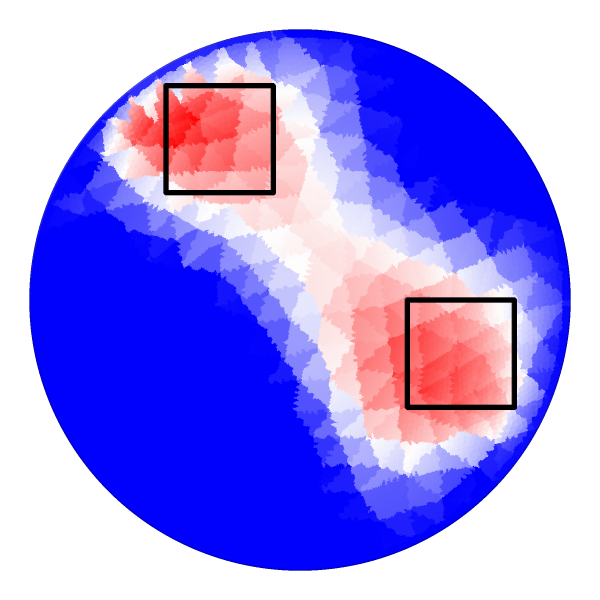}&
\includegraphics[width = \sevenfiglen, trim = {0.25cm, 0.25cm, 0.25cm, 0.25cm}, clip]{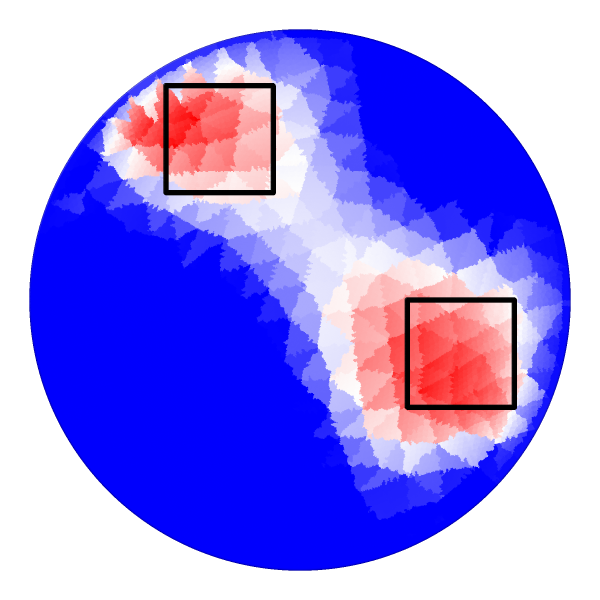}\\
(a) $\Gamma_D$ & (b) $k=0$ & (c) $k=20$ & (d) $k=30$
\end{tabular}
\caption{
The reconstructions of modulus nonlinear potential inclusions for Example \ref{exam5}. \label{fig5} }
\end{figure}

The IDSM provides an accurate localization of both inclusions, cf. Fig. \ref{fig5}(b)-(d), despite inherent nonlinearity of the PDE, nonsmoothness of the potential term $u|y|y$, limited data from a single measurement pair, and partial boundary coverage.
The results show the robustness of the IDSM in reconstructing inclusions in nonlinear systems with only one pair of Cauchy data.

\subsection{Damping factor}\label{subsec:pltLambda}

Now we present the evolution of the damping factor $\frac{1}{1+\lambda_{k-1,p}}$.
The results for Examples \ref{exam1}, \ref{exam2}, \ref{exam4}, and \ref{exam5} (with $p=2$) are shown in Fig. \ref{fig0}(a), and Example \ref{exam3} (with $p_1=1$ and $p_2=99$) in Fig. \ref{fig0}(b).

\begin{figure}[hbt!]
\centering
\begin{tabular}{cc}
\includegraphics[width = 0.45\textwidth]{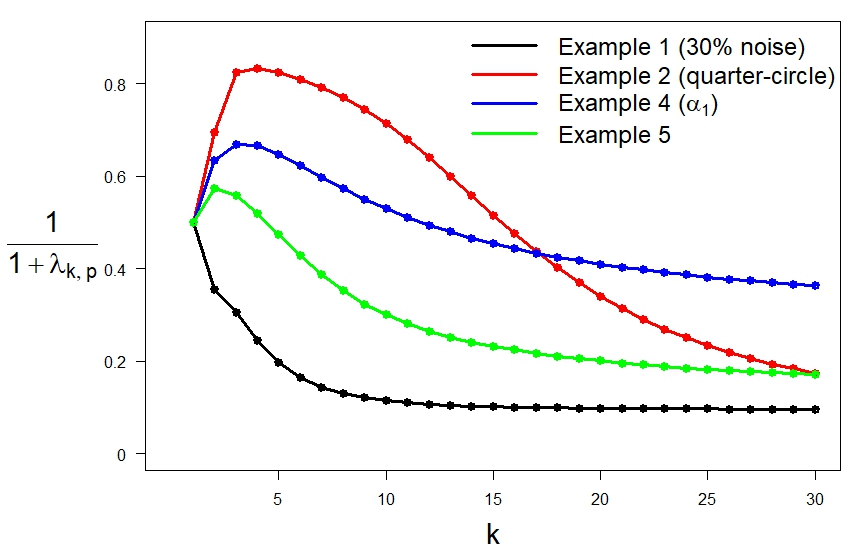}&
\includegraphics[width = 0.45\textwidth]{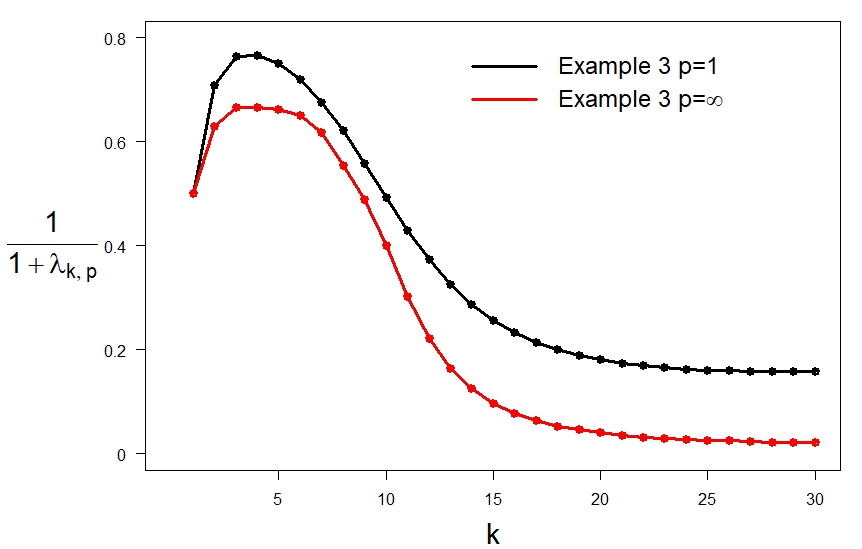}\\
Examples \ref{exam1}, \ref{exam2}, \ref{exam4}, and \ref{exam5}&Example \ref{exam3}
\end{tabular}
\caption{
The evolution of the damping factor $\frac{1}{1+\lambda_{k-1,p}}$ versus the iteration count $k$.
\label{fig0}
}
\end{figure}

For the integrability index $p=2$, the damping factor $\lambda_{k-1,p}$ exhibits a distinct U shape (Fig. \ref{fig0} (a)), which aligns with the convergence of the algorithm. We
initialize $\lambda$ at $\lambda_{0,p} = 1$, giving a damping factor $\frac{1}{1+\lambda_{0,p}} = 0.5$, which effectively suppresses the influence of inaccurate low-rank terms.
When the approximation improves, $\lambda_{k-1,p}$ decreases monotonically, and the damping factor tends to unity ($\frac{1}{1+\lambda_{k-1,p}} \to 1^{-}$) to preserve high-fidelity corrections.
When the algorithm converges, $\lambda_{k-1,p}$ increases again due to the decay in the duality pairing $\langle \hat{\zeta}^{k+1}, \hat{\eta}^{k+1} \rangle$.

In Example \ref{exam3} (Fig. \ref{fig0}, right), the damping factor scales inversely with $p$: a small $p=1$ yields larger asymptotic damping ($\frac{1}{1+\lambda_{k-1,p}} \to 0.9$), while a large $p=99$ (approximating $\infty$) results in small damping ($\frac{1}{1+\lambda_{k-1,p}} \to 0.1$).
This aligns with the remarks in Section \ref{subsec:sroc}: a larger $p$ tightens the spectral bound on $\widetilde{\mathcal{R}}^{k}$ and leads to stronger damping for low-rank terms.

The stabilization of the damping factor post-U-turn provides a natural stopping criterion.
This corroborates the theoretical analysis in Section \ref{subsec:sroc}: the uniform spectral boundedness of $\widetilde{\mathcal{R}}^{k}$ ensures noise suppression and prevents uncontrolled error propagation.
Thus, the stabilization-correction scheme can balance  stability and accuracy.

\section{Conclusion}\label{sec:CON}

We have developed a novel IDSM for elliptic inverse problems with partial Cauchy data.
The key innovations include: the data completion strategy, the heterogeneously regularized DtN map, and a stabilization-correction scheme. These innovations enhance the near-orthogonality of probing functions, suppress spurious artifacts, and guarantee stability even under severe noise.
The method is applicable to diverse elliptic inverse problems, including nonlinear and mixed-type models. The stability is established through spectral analysis of the resolver. Numerical experiments shows its efficiency (30–100 PDE solves) and robustness to noise (up to 30\%), boundary configurations, and parameter choices.
\bibliographystyle{siam}
\bibliography{ref}

\begin{thebibliography}{10}

\bibitem{Ammari2008}
{\sc H.~Ammari}, {\em An Introduction to Mathematics of Emerging Biomedical Imaging}, Springer, Berlin, 2008.

\bibitem{AmmariCalmon:2008}
{\sc H.~Ammari, P.~Calmon, and E.~Iakovleva}, {\em Direct elastic imaging of a small inclusion}, SIAM J. Imaging Sci., 1 (2008), pp.~169--187.

\bibitem{AmmariLesselier:2005}
{\sc H.~Ammari, E.~Iakovleva, and D.~Lesselier}, {\em A {MUSIC} algorithm for locating small inclusions buried in a half-space from the scattering amplitude at a fixed frequency}, Multiscale Model. Simul., 3 (2005), pp.~597--628.

\bibitem{Chow2021}
{\sc Y.~T. Chow, F.~Han, and J.~Zou}, {\em A direct sampling method for simultaneously recovering inhomogeneous inclusions of different nature}, SIAM J. Sci. Comput., 43 (2021), pp.~B678--B711.

\bibitem{Zou2015}
{\sc Y.~T. Chow, K.~Ito, K.~Liu, and J.~Zou}, {\em Direct sampling method for diffusive optical tomography}, SIAM J. Sci. Comput., 37 (2015), pp.~A1658--A1684.

\bibitem{Zou2014}
{\sc Y.~T. Chow, K.~Ito, and J.~Zou}, {\em A direct sampling method for electrical impedance tomography}, Inverse Problems, 30 (2014), p.~095003.

\bibitem{Kirsch1996}
{\sc D.~Colton and A.~Kirsch}, {\em A simple method for solving inverse scattering problems in the resonance region}, Inverse Problems, 12 (1996), pp.~383--393.

\bibitem{Schnabel1996}
{\sc J.~E. Dennis and R.~B. Schnabel}, {\em Numerical Methods for Unconstrained Optimization and Nonlinear Equations}, SIAM, Philadelphia, PA, 1996.

\bibitem{ElBadia2002}
{\sc A.~El~Badia and T.~Ha-Duong}, {\em On an inverse source problem for the heat equation: Application to a pollution detection problem}, J. Inverse Ill-Posed Probl., 10 (2002), pp.~585--599.

\bibitem{Scacchi2014}
{\sc P.~Franzone, L.~Pavarino, and S.~Scacchi}, {\em {Mathematical Cardiac Electrophysiology}}, Springer, Heidelberg, 2014.

\bibitem{Devaney2004}
{\sc F.~K. Gruber, E.~A. Marengo, and A.~J. Devaney}, {\em Time-reversal imaging with multiple signal classification considering multiple scattering between the targets}, J. Acoust. Soc. Am., 115 (2004), pp.~3042--3047.

\bibitem{Ito2025Iterative}
{\sc K.~Ito, B.~Jin, F.~Wang, and J.~Zou}, {\em Iterative direct sampling method for elliptic inverse problems with limited cauchy data}, SIAM J. Imaging Sci., 18 (2025), pp.~1284--1313.

\bibitem{ItoJinZou:2012}
{\sc K.~Ito, B.~Jin, and J.~Zou}, {\em A direct sampling method to an inverse medium scattering problem}, Inverse Problems, 28 (2012), p.~025003.

\bibitem{jin2025adaptive}
{\sc B.~Jin, F.~Wang, and Y.~Xu}, {\em Adaptive approximations of inclusions in a semilinear elliptic problem related to cardiac electrophysiology}, IMA J. Numer. Anal., 45 (2025), pp.~1--41.

\bibitem{Kirsch1998}
{\sc A.~Kirsch}, {\em Characterization of the shape of a scattering obstacle using the spectral data of the far field operator}, Inverse Problems, 14 (1998), pp.~1489--1512.

\bibitem{Wright2006}
{\sc J.~Nocedal and S.~Wright}, {\em {Numerical Optimization}}, Springer, New York, 2006.

\bibitem{Novotny2006}
{\sc L.~Novotny and B.~Hecht}, {\em Principles of Nano-Optics}, Cambridge University Press, Cambridge, 2006.

\bibitem{Potthast1998}
{\sc R.~Potthast}, {\em A point source method for inverse acoustic and electromagnetic obstacle scattering problems}, IMA J. Appl. Math., 61 (1998), pp.~119--140.

\bibitem{Sun2023}
{\sc S.~Sun, G.~Nakamura, and H.~Wang}, {\em Numerical studies of domain sampling methods for inverse boundary value problems by one measurement}, J. Comput. Phys., 485 (2023), p.~112099.

\bibitem{Yakowitz1980}
{\sc S.~Yakowitz and L.~Duckstein}, {\em Instability in aquifer identification: Theory and case studies}, Water Resour. Res., 16 (1980), pp.~1045--1064.

\bibitem{Zhdanov2015}
{\sc M.~S. Zhdanov}, {\em Inverse Theory and Applications in Geophysics}, Elsevier, Amsterdam, 2015.

\end{thebibliography}
\end{document}